\documentclass[11pt, a4paper, reqno]{amsart}
\usepackage{amsmath}
\usepackage{amsfonts}
\usepackage{amssymb}
\usepackage{color}
\usepackage{comment,graphicx}
\usepackage[T1]{fontenc}
\usepackage{url}
\usepackage{ae}

\usepackage{xcolor}
\usepackage{hyperref}
\hypersetup{
colorlinks=true,
linkcolor=blue,
citecolor=blue
}

\def\phi{\varphi}
\def\rho{\varrho}
\def\epsilon{\varepsilon}
\numberwithin{equation}{section}
\theoremstyle{plain}
\newtheorem{theorem}{Theorem}[section]
\newtheorem{lemma}[theorem]{Lemma}

\newtheorem{corollary}[theorem]{Corollary}
\newtheorem{definition}[theorem]{Definition}

\newtheorem{remark}[theorem]{Remark}

\renewcommand{\le}{\leqslant}
\renewcommand{\ge}{\geqslant}
\renewcommand{\leq}{\leqslant}
\renewcommand{\geq}{\geqslant}

\title[On Convolution in Variable Lebesgue Spaces and Fractional NSE]{On Convolution in Variable Lebesgue Spaces and Applications to Fractional Navier–-Stokes Equations}
\author[ S. BenMahmoud]{ Salah BenMahmoud$^{\ 1,2}$}
\date{February, 21. 2026 }

\keywords{Fractional Navier-Stokes equations; variable exponent Lebesgue spaces; mixed-norm spaces; space-time estimates; convolution-type inequalities; well-posedness.\\
\\
  {\footnotesize Salah BenMahmoud} \\
 $\,^{1}${\tiny Department of Computer Science, Faculty of Science and Technology, University of Ahmed Zabana, Relizane 48000, Algeria.}\\
 $\,^{2}${\tiny Laboratory of Functional Analysis and Geometry of Spaces, University of Mohamed Boudiaf, M'sila 28000, Algeria.}\\
\vspace{1cm} {\tiny Emails : \textit{salahmath2016@gmail.com / salah.benmahmoud@univ-relizane.dz}}
}

\begin{document}
\maketitle

\begin{abstract}
In this paper, we introduce a new class of convolution-type inequalities in variable exponent Lebesgue spaces and derive several related estimates, including  the  \(L^{r(\cdot)}\)--\(L^{p(\cdot)}\) smoothing estimate for the fractional heat kernel. We demonstrate the usefulness of these inequalities by establishing the local well-posedness results for  mild solutions to  the fractional Navier--Stokes equations, and we further extend these results to global-in-time well-posedness for sufficiently small initial data. Our analysis is carried out in a wide range of mixed-norm variable exponent Lebesgue spaces, including the fully variable setting $L^{p(\cdot)}_t L^{q(\cdot)}_x$, highlighting the robustness of the proposed framework under non-constant integrability. Moreover, the proposed framework is expected to serve as a key tool for similar applications in other related variable exponent function spaces.
\newline


\end{abstract}

\section{Introduction}
In this paper, we present new results related to convolution-type inequalities in
variable Lebesgue spaces and prove their relevance to existence results for the
fractional Navier--Stokes equations. The convolution operator takes two measurable
functions \(u\) and \(v\) defined on \(\mathbb{R}^n\) and produces the function defined by
\begin{equation*}
u*v(x)=\int_{\mathbb{R}^n}\,u(x-t)v(t)\,dt ,\, x\in \mathbb{R}^n.
\end{equation*}
It appears in different areas of mathematics, including the theory of
partial differential equations and integral equations. The boundedness of the
convolution operator on various function spaces has attracted significant attention,
since it is a key tool, in particular, in Lebesgue spaces with constant exponents,
the classical Young's inequality states that
 \[
\|u * v\|_{L^r(\mathbb{R}^n)} \le \|u\|_{L^p(\mathbb{R}^n)} \, \|v\|_{L^q(\mathbb{R}^n)},
\]
if $\frac{1}{p} + \frac{1}{q} - 1 = \frac{1}{r}$. This result has been a central tool of great importance in solving many problems in mathematical analysis. One of these problems is the system of fractional incompressible Navier--Stokes equations, which is defined in \(\mathbb{R}^3\) as follows,
\begin{equation}\label{NS_Intro}
\begin{cases}
\partial_t u=- (-\Delta)^{\alpha}  u-( u \cdot \nabla u)  u+ \nabla P + f, \qquad \\[3pt]
\operatorname{div}( u)=0,\\[3pt]
 u(0,x)= u_0(x),\quad \operatorname{div}( u_0)=0, \qquad x\in \mathbb{R}^3,
\end{cases}\tag{FNSE}
\end{equation}
with $\alpha\in (\frac{1}{2},1]$ and $(-\Delta)^{\alpha}$
is the fractional Laplacian operator, which can be defined by the Fourier transform,
$
(-\Delta)^{\alpha} u = \mathcal{F}^{-1}\!\left[\,|\xi|^{2\alpha}\, \mathcal{F}u(\xi, t)\right]
$, 
$ u:[0, +\infty)\times \mathbb{R}^3 \rightarrow \mathbb{R}^3$ the vector field denoting the velocity of a viscous, incompressible and homogeneous fluid,   $P:[0, +\infty)\times\mathbb{R}^3 \rightarrow  \mathbb{R}$ is its pressure and $ u_0:\mathbb{R}^3 \rightarrow \mathbb{R}^3$, $f:[0, +\infty)\times \mathbb{R}^3 \rightarrow \mathbb{R}^3$ are a given initial data and a given external force, respectively. In the present work, we focus on the fractional incompressible Navier–-Stokes equations in $\mathbb{R}^3$ for the dissipation range $\alpha\in (\frac{1}{2},1]$. While our primary analysis is situated in three dimensions and the results can largely be extended to $\mathbb{R}^n$ for $\alpha > 0$, certain estimates established in $\mathbb{R}$ do not naturally extend to $\mathbb{R}^2$ or $\mathbb{R}^3$. Consequently, we treat the one-dimensional case separately in specific instances. Note that  when $\alpha = 1$, the system  (\ref{NS_Intro}) reduces to the classical Navier–-Stokes equations. 

The convolution arises from the usual fractional heat kernel \(\mathcal{F}^{-1}\!\bigl[e^{-t|\xi|^{2\alpha}}\bigr]\), since the fundamental convolution identity reduces to $\mathcal{F}^{-1} [ e^{-t|\xi|^{2\alpha}} \mathcal{F} u(\xi,t) ] = ( \mathcal{F}^{-1} e^{-t|\xi|^{2\alpha}} ) * u(x) 
$. The Navier--Stokes equations involve evolution in both time \(t\) and space \(x\), hence
the natural framework for investigating the existence and regularity of solutions is the class of mixed-norm Lebesgue spaces, usually denoted by \(L^q_t L^p_x\),  \(L^q((0,+\infty); L^p(\Omega))\) or \(L^q(0,T; L^p(\Omega))\), \(T \in (0,+\infty]\). Similarly, the literature reflects an extensive use of broader mixed-norm frameworks, including Sobolev, Besov, and Triebel-Lizorken spaces. Furthermore, mixed-norm Lorentz spaces, denoted by \(L^{q,r}_t\, L^p_x\), have become indispensable for tracking the fine-grained regularity of solutions to fluid dynamics equations, see for example \cite{KozonoShimizu2018}.

A vast literature exists for the classical Navier--Stokes equations where convolution-type estimates and mixed-norm spaces play central roles, this line of research began with the pioneering work of Kato and Fujita 
 \cite{KatoFujita1962} in 1962, who introduced the semigroup approach to establish the existence of strong solutions. More recently, this framework has been extended in several notable works, including\cite{ChangJin2016} and  \cite{DD}  where Navier–-Stokes equations where studied in  $L^\infty(0, T; L^n(\mathbb{R}^n))$ and   $L^q_\alpha\,L^p$, respectively. In \cite{LiZhangZhang2021} the authors utilize a variety of mixed-norm spaces such $L^\infty((0, \infty), L^p(\mathbb{R}^3))$ and Lorentz-Mixed Spaces $L^q((0, \infty), L^{p, \infty}(\mathbb{R}^3))$. For classical results on the incompressible Navier--Stokes equations. For an exhaustive treatment and a comprehensive historical trajectory of the Navier–Stokes equations, we refer the reader to the  monographs of Lemarié-Rieusset \cite{LemarieRieusset2002, Lemarie-Rieusset2024}.

In recent years, authors have been interested in developing the tools and techniques to study various partial differential equations within the framework of variable Lebesgue spaces, denoted by \(L^{p(x)}\), and related function spaces such as variable Sobolev spaces \(W^{m,p(x)}\) (see \cite[Part II]{DHHR} and \cite[Chapter 6]{CF}), variable Besov spaces \(B_{p(\cdot),q(\cdot)}^{s(\cdot)}\), and variable Triebel--Lizorkin spaces \(F_{p(\cdot),q(\cdot)} ^{s(\cdot)}\) (see \cite{AH,DHR}). Variable Lebesgue spaces are considered a generalization of classical Lebesgue spaces. However, they present many challenges  and lack some properties -see \cite{DieningHastoNekvinda2005}- compared to the classical spaces such as the failure of translation invariance and other notable issue is the failure of Young's inequality. In other words, even if
$
\frac{1}{p(x)} + \frac{1}{q(x)} - 1 = \frac{1}{r(x)}
\quad \text{a.e. in } \mathbb{R}^n,
$
the estimate
\[
\|f * g\|_{L^{r(\cdot)}(\mathbb{R}^n)}
\le C \,
\|f\|_{L^{p(\cdot)}(\mathbb{R}^n)}
\|g\|_{L^{q(\cdot)}(\mathbb{R}^n)}
\]
does not hold. Nevertheless, under additional assumptions on the exponents \(p(\cdot)\), \(q(\cdot)\), and \(r(\cdot)\), restricted versions of  convolution-type estimates remain valid. This issue raises significant difficulties in studying partial differential equations in variable Lebesgue spaces and related function spaces, including the mixed-norm variable Lebesgue spaces, denoted by \(L^{q(\cdot)}_t L^{p(\cdot)}_x\). Despite these challenges, there exist many interesting results, some notable examples include
\cite{Chamorro2025,DudekSkrzypczak2017,Jiang2015,MihailescuRadulescu2012,Phan2019}.

In the papers \cite{GVH2} and \cite{GVH1}, local and global mild existence results were
established for (\ref{NS_Intro}) in the mixed spaces
\(L^{p(t)}\big(0,T; L^{q}(\mathbb{R}^{3})\big)\), where \(q\) is a constant exponent.
The restriction to constant \(q\) arises from the failure of Young's inequality in
\(L^{p(\cdot)}(\mathbb{R}^n)\) and the current lack of general convolution-type estimates in variable exponent spaces. These results later served as a basis in \cite{CVZ} to
establish local existence and regularity of solutions to the two-dimensional
dissipative surface quasi-geostrophic equation in variable Lebesgue spaces , see also the references therein for further related results. a further interesting development was presented in \cite{Sun2025}. In this work, the authors investigated the well-posedness of the generalized magnetohydrodynamic (MHD) equations, extending the framework to product spaces of the form $L^{p(\cdot)}(0, T; L^{p}(\mathbb{R}^{3})) \times L^{q(\cdot)}(0, T; L^{q}(\mathbb{R}^{3}))$. Furthermore, they established existence results in the intersection of mixed-norm spaces $L^{p(\cdot)}(\mathbb{R}^{3}; L^{\infty}(0, +\infty)) \cap L^{\frac{3}{2\alpha-1}}(\mathbb{R}^{3}; L^{\infty}(0, +\infty))$.

  Nevertheless, these results are generally confined to spaces of the form $L^{p(\cdot)}_t L^q_x$, where the spatial exponent $q$ remains constant, in the present work, we bridge this gap by establishing the first systematic treatment of the fully variable mixed-norm case $L^{p(\cdot)}_t L^{q(\cdot)}_x$. 

We aim to present techniques for working in mixed-norm variable Lebesgue spaces
\(L^{p(\cdot)}_t L^{q(\cdot)}_x\), an area that has been rapidly developing in recent years.
Despite their growing importance, the analysis of partial differential equations in
these function spaces still faces significant technical obstacles, including those
related to convolution estimates. Our goal is to address some of these difficulties
and provide tools that facilitate the study of PDEs within this framework. As an
application of the effectiveness of these techniques, we establish the existence of
mild solutions to the fractional Navier--Stokes  equations (\ref{NS_Intro}) in a variety of mixed-norm
variable Lebesgue spaces.  These methods and results  are highly adaptable and expected to yield similar advancements in the context of mixed-norm variable exponent scales including \(W^{m,p(x)}\), \(B_{p(\cdot),q(\cdot)}^{s(\cdot)}\) and  \(F_{p(\cdot),q(\cdot)} ^{s(\cdot)}\).

The main goals of this paper are twofold. First, we establish convolution-type
estimates presented in Lemmas \ref{LemmaCom 1}, \ref{LemmaCom 2}, \ref{Lemma of LV cap Lp} and
\ref{Lemma of L2 cap Lp}, as well as in Theorems \ref{THM 1}, \ref{THM1-2}, \ref{Prop2}. These results include estimates of the forms
$
\|\eta_{t,m} * u\|_{L^{p(x)}(\mathbb{R}^n)} \leq Ct^{-n\vartheta(t)} \|u\|_{L^{p(x)/2}(\mathbb{R}^n)}
$ and 
$
\|\eta_{t,m} * u\|_{L^{r(x)}(\mathbb{R}^n)} \leq Ct^{-n\omega(t)} \|u\|_{L^{p(x)}(\mathbb{R}^n)}
$
which relate the spaces  \(L^{p(x)}(\mathbb{R}^n),L^{r(x)}(\mathbb{R}^n) \) and \(L^{\frac{p(x)}{2}}(\mathbb{R}^n)\).
Such estimates constitute a key tool in the analysis of fractional Navier--Stokes equations
 and are particularly useful for handling nonlinear terms involving convolution operators.
  Some of our results cover the cases where \(p^{+} = \infty\), moreover, the exponents \(p\) 
  or \(q\) may take the value \(\infty\) on subsets of \(\mathbb{R}^n\), which introduces further
technical complications and additional analytical challenges.

The second goal of this work is to apply these convolution-type estimates 
to establish the fundamental smoothing estimates for the fractional heat semigroup in variable Lebesgue spaces presented in Theorems \ref{Prop - 1 - app}, \ref{Prop - 2 - app} and  \ref{Prop - 3 - app} and proving the existence of local and global mild solutions to the fractional Navier--Stokes
equations \eqref{NS_Intro}, as stated in Theorems 
\ref{THM1 - Local - Existence}--\ref{THMFinall - Local - Existence},  \ref{THM glob 1} and \ref{THM glob 2}. The proofs are presented in a clear and detailed manner, with each step carefully explained to ensure readability and accessibility.

The remainder of this paper is organized as follows. In Section~\ref{sec-pre}, we
introduce the notation and present definitions and basic results concerning variable
Lebesgue spaces, together with the necessary preliminaries. Section~\ref{sec-conv}
is devoted to convolution-type estimates in variable Lebesgue spaces, where we
establish results analogous to Young's inequality and several lemmas providing
convolution-type estimates that play a crucial role in proving existence theorems
for the fractional Navier--Stokes equations in the subsequent sections. Section~\ref{Sec-Exis} is divided into three subsections. We begin by recalling
preliminary results on the fractional Navier--Stokes equations, including the
usual fractional heat kernel
\(\mathcal{F}^{-1}\!\bigl[e^{-t|\xi|^{2\alpha}}\bigr]\) and the kernel associated
with the operator \(e^{-t(-\Delta)^{\alpha}} \mathbb{P}\operatorname{div}(\cdot)\),
 where  \( e^{-t(-\Delta)^{\alpha}} \) denotes the semigroup generated by the fractional Laplacian $(-\Delta)^{\alpha}$.  Then in Subsection~\ref{Est - HEAT semigroup} we derive Theorems~\ref{Prop - 1 - app}, \ref{Prop - 2 - app} and  \ref{Prop - 3 - app}, which states that, under certain conditions,  the norm of the heat semigroup $\| e^{-t(\Delta)^\alpha} u \|_{L^{p(\cdot)}(\mathbb{R}^n)}$ is bounded by
\[
 t^{-n\vartheta(t)}\, \|u\|_{L^{r(\cdot)}(\mathbb{R}^n)} \, \text{ or } \,
 t^{-n\vartheta(t)} \|u\|_{L^{r(\cdot)}(\mathbb{R}^n)\cap L^{\nu}(\mathbb{R}^n)},
 \, t>0,
\]
where $\vartheta(t)$ is a decay function determined by $\alpha,\nu$ and the variable exponents $p(\cdot)$ and $r(\cdot)$, which can be viewed as a generalization of the
classical \(L^r\)--\(L^p\) smoothing estimate for the heat semigroup. 
 In subsection \ref{Sec-Exis-local} we prove local existence theorems, this resuls are presented for sufficiently small $T\in(0,1]$ in  Theorems \ref{THM1 - Local - Existence}, \ref{THM1.1 - Local - Existence}, \ref{THM4 - Local - Existence} for the space $L^{p(\cdot)}\big(0,T;L^{q(\cdot)}(\mathbb{R}^{3})\big)$ , in Theorem  \ref{THM2 - Local - Existence} for the space $L^{p(\cdot)}\big(0,T;L^{q(\cdot)}(\mathbb{R}^{3})\cap L^\infty(\mathbb{R}^3)\big)$  and in Theorem \ref{THM3 - Local - Existence} for the space $L^{p(\cdot)}\big(0,T;L^{q(\cdot)}(\mathbb{R}^3)\cap L^\nu(\mathbb{R}^3)\big)$ and in Theorem~\ref{THMFinall - Local - Existence} for the space  $L^{p(\cdot)}\big(0,T;L^{q(\cdot)}(\mathbb{R})\cap L^2(\mathbb{R})\big)$. In the final Subsection~\ref{sec-exist-glob-exist}, we prove global existence
theorems for small initial data in the spaces
$
L^{p(\cdot)}((0,+\infty); L^{q(\cdot)}(\mathbb{R}^{3}) \cap L^\infty(\mathbb{R}^3))$, 
$
L^{p(\cdot)}((0,+\infty); L^{q}(\mathbb{R}^{3}))$
 and  $
L^{p}(0,T; L^{q(\cdot)}(\mathbb{R}^{3})),\ T >0$ stated in Theorems \ref{THM glob 1} and \ref {THM glob 2} respectively, 
we also establish corresponding global existence results on bounded intervals \((0,T)\), see Remark \ref{finalRamrk}.
\bigskip
\section{Preliminaries}\label{sec-pre}
Now, we present some notations. As usual, we denote by $\mathbb{R}^{n}$ the $n$-dimensional real Euclidean
space.   If $\Omega \subset {\mathbb{R}^{n}}$ is a measurable set, then $|\Omega|$ stands for
the Lebesgue measure of $\Omega$ and $\chi _{\Omega}$ denotes its characteristic
function. By $c$ or $C$ we denote generic positive constants, which may have
different values at different occurrences. Although the exact values of the
constants are usually irrelevant for our purposes, sometimes we emphasize
their dependence on certain parameters (e.g., $c(p), C(p)$ means that $c$ and $C$ depend
on $p$, etc.). We write  $f \leq Cg$($f \leq c\ g$) to signify that the inequality holds for some constant $C > 0$ independent of the (non-negative) functions $f$ and $g$. 

Let $X$ and $Y$ be normed spaces that are continuously embedded into a Hausdorff topological vector space $Z
$. We define their intersection $
X \cap Y := \{ f \in Z \mid f \in X \text{ and } f \in Y \},
$ and equip it with the norm $\|f\|_{X \cap Y} := \max\{ \|f\|_X, \|f\|_Y \}$. 
We define the Fourier transform of a
function $f\in L^1(\mathbb{R}^{n})$ by 
\begin{equation*}
\mathcal{F}(f)(\xi ):=(2\pi )^{-n/2}\int_{\mathbb{R}^{n}}e^{-ix\cdot \xi
}f(x)dx,\quad \xi \in \mathbb{R}^{n}.
\end{equation*}

 The variable exponents that we consider are always measurable functions $p$ on $\mathbb{R}^{n}$ with range
in $[1,\infty ]$, we denote the set of all such
functions by $\mathcal{P}(\mathbb{R}^{n})$. For $p\in \mathcal{P}(\mathbb{R}^{n})$
the conjugate exponent of $p$ denoted by ${p}^{\prime }$ 
 is  given by $\frac{1}{{p(\cdot )}}+\frac{1}{{p}%
^{\prime }{(\cdot )}}=1$ with the convention  $\frac{1}{\infty} = 0$.
 We use the standard notations: 
\begin{equation*}
p^{-}:=\underset{x\in \mathbb{R}^{n}}{\text{ess-inf}}\,p(x)\quad \text{and}%
\quad p^{+}:=\underset{x\in \mathbb{R}^{n}}{\text{ess-sup}}\,p(x).
\end{equation*}%
The function spaces in this paper are fit into the framework of  semi-modular spaces, see for example \cite[Chapter 2]{DHHR} and \cite[Section 2.10]{CF}. The function $\omega _{p}$ is defined as follows:
\begin{equation*}
\omega _{p}(t)=\left\{ 
\begin{array}{ccc}
t^{p} & \text{if} & p\in \lbrack 1,\infty )\text{ and }t>0, \\ 
0 & \text{if} & p=\infty \text{ and }0<t\leq 1, \\ 
\infty & \text{if} & p=\infty \text{ and } t>1.\qquad
\end{array}
\right.
\end{equation*}%

The convention $1^\infty=0$ is adopted in order that $\omega _{p}$ be left-continuous. The variable exponent modular is defined by 
\begin{equation*}
\varrho_{p(\cdot )}(f):=\int_{\mathbb{R}^{n}}\omega _{p(x)}(|f(x)|)\,dx.
\end{equation*}%
For simplicity, we shall write $\omega_{p(\cdot)}(|f(x)|)=|f(x)|^{p(x)}$, 
keeping in mind that this notation also covers the case \(p(x)=\infty\). The variable exponent Lebesgue space denoted $L^{p(\cdot )}(\mathbb{R}^{n})$ or $L^{p(x )}(\mathbb{R}^{n})$\ consists of measurable
functions $f$ on $\mathbb{R}^{n}$ such that $\varrho _{p(\cdot )}(\lambda
f)<\infty $ for some $\lambda >0$. We define the Luxemburg norm on
this space by the formula 
\begin{equation*}
\left\Vert f\right\Vert _{L^{p(\cdot )}(\mathbb{R}^{n})}:=\inf \Big\{\lambda >0:\varrho
_{p(\cdot )}\Big(\frac{f}{\lambda }\Big)\leq 1\Big\}.
\end{equation*}%
since the exponents considered in this paper are of range $[1,+\infty]$ then $f\to \varrho _{p(\cdot )}(f)$  is convex, see \cite[Section 3.1]{DHHR} and \cite[Proposition 2.7]{CF}, a consequence of the convexity of $\varrho_{p(\cdot )}$ is that if $\alpha > 1$, then $\alpha \varrho_{p(\cdot )}(f) \leq \varrho_{p(\cdot )}(\alpha f)$, and if $0 < \alpha < 1$, then $\varrho_{p(\cdot )}(\alpha f) \leq \alpha \varrho_{p(\cdot )}(f)$.
Another useful property, often referred to as the unit ball property, states that   $\left\Vert f\right\Vert _{L^{p(\cdot )}(\mathbb{R}^{n})}\leq 1$ if and only if $\varrho _{p(\cdot )}(f)\leq 1$, see \cite[Lemma 3.2.4]{DHHR}.
These properties will be invoked frequently throughout our analysis, often without explicit mention, as they are fundamental to the scaling of variable exponent spaces.

Let \(\Omega \subset \mathbb{R}^n\) and let \(f\) be measurable on \(\Omega\).
Then \(f \in L^{p(\cdot )}(\Omega)\) if and only if $\chi_{\Omega}\,f \in L^{p(\cdot )}(\mathbb{R}^n)$, 
$\varrho _{p(\cdot ),\Omega}(f):=\varrho _{p(\cdot )}(\chi_{\Omega}\,f)$  and 
$\left\Vert f\right\Vert _{L^{p(\cdot )}(\Omega)}:=\left\Vert \chi_{\Omega}\,f \right\Vert _{L^{p(\cdot )}(\mathbb{R}^{n})}$. Variable exponent Lebesgue spaces $L^{p(\cdot)}(\mathbb{R}^n)$ are Banach spaces that generalize the classical $L^p(\mathbb{R}^n)$ framework while retaining several fundamental structural properties. In particular, the variable exponent version of H\"{o}lder’s inequality holds, by \cite[Lemma 3.2.20]{DHHR}, if $ p,q,s\in \mathcal{P}(\mathbb{R}^{n})$ are such that $\frac{1}{s}=\frac{1}{p}+\frac{1}{q}$, then  for every $f\in L^{p(\cdot)}$ and $g\in L^{q(\cdot)}$
\begin{equation*}
\| f\, g\big\|_{L^{s(\cdot )}(\mathbb{R}^{n})}\le C \| f\big\|_{L^{p(\cdot )}(\mathbb{R}^{n})}\| g\big\|_{L^{q(\cdot )}(\mathbb{R}^{n})}. \label{hold-in}
\end{equation*} 
Now,  we state  a version of Minkowski's integral inequality for variable Lebesgue spaces, see \cite[Corollary 2.38]{CF}. 
If  $p(\cdot) \in \mathcal{P}(\Omega)$ and  $f: \mathbb{R}^n \times \mathbb{R}^n \to \mathbb{R}$ is a measurable function (with respect to the product measure) such that for almost every $y \in \Omega$, $f(\cdot, y) \in L^{p(\cdot)}(\mathbb{R}^n)$. Then,
\begin{equation*}\label{cor:minkowski}
\left\| \int_{\mathbb{R}^n} f(\cdot, y) \, dy \right\|_{L^{p(\cdot)}(\mathbb{R}^n)} \leq C \int_{\mathbb{R}^n} \| f(\cdot, y) \|_{L^{p(\cdot)}(\mathbb{R}^n)} \, dy,
\end{equation*}
where the positive constant  $C$ depends on the exponent $p(\cdot)$.

We say that a real valued-function $g$ on $\mathbb{R}^{n}$ is \textit{%
locally }log\textit{-H\"{o}lder continuous} on $\mathbb{R}^{n}$, abbreviated 
$g\in C_{\text{loc}}^{\log }(\mathbb{R}^{n})$, if there exists a constant $%
c_{\log }(g)>0$ such that 
\begin{equation}
\left\vert g(x)-g(y)\right\vert \leq \frac{c_{\log }(g)}{\log
(e+1/\left\vert x-y\right\vert )}  \label{lo-log-Holder}
\end{equation}%
for all $x,y\in \mathbb{R}^{n}$.

We say that $g$ satisfies the log\textit{-H\"{o}lder decay condition}, if
there exist two constants $g_{\infty }\in \mathbb{R}$ and $c_{\log }>0$ such
that%
\begin{equation*}
\left\vert g(x)-g_{\infty }\right\vert \leq \frac{c_{\log }}{\log
(e+\left\vert x\right\vert )}
\end{equation*}%
for all $x\in \mathbb{R}^{n}$. We say that $g$ is \textit{globally} log%
\textit{-H\"{o}lder continuous} on $\mathbb{R}^{n}$, abbreviated $g\in
C^{\log }(\mathbb{R}^{n})$, if it is\textit{\ }locally log-H\"{o}lder
continuous on $\mathbb{R}^{n}$ and satisfies the log-H\"{o}lder decay\textit{%
\ }condition.\textit{\ }The constants $c_{\log }(g)$ and $c_{\log }$ are
called the \textit{locally }log\textit{-H\"{o}lder constant } and the log%
\textit{-H\"{o}lder decay constant}, respectively\textit{.} We note that any
function $g\in C_{\text{loc}}^{\log }(\mathbb{R}^{n})$ always belongs to $%
L^{\infty }(\mathbb{R}^{n})$.

We define the following class of variable exponents: 
\begin{equation*}
\mathcal{P}^{\mathrm{log}}(\mathbb{R}^{n}):=\Big\{p\in \mathcal{P}(
\mathbb{R}^{n}):\frac{1}{p}\in C^{\log }(\mathbb{R}^{n})\Big\},
\end{equation*}%
 which is the standard class of exponents for which the Hardy--Littlewood maximal operator is bounded, see~\cite[Chapter 4]{DHHR}. We define 
\begin{equation*}
\frac{1}{p_{\infty }}:=\lim_{|x|\rightarrow \infty }\frac{1}{p(x)},
\end{equation*}%
and we use the convention $\frac{1}{\infty }=0$. Note that although $\frac{1%
}{p}$ is bounded, the variable exponent $p$ itself can be unbounded. We define the least bell-shaped majorant 
\(\Psi\) of  a function $\varphi \in L^{1}(\mathbb{R}^{n})$ by
\begin{equation*}
\Psi \left( x\right) :=\sup_{\left\vert y\right\vert \geq \left\vert
x\right\vert }\left\vert \varphi \left( y\right) \right\vert \, , \,  x \in \mathbb{R}^{n}\, .
\end{equation*}%

We suppose that $\Psi \in L^{1}(\mathbb{R}^{n})$, then it is
proved in $\text{\cite[Lemma \ 4.6.3]{DHHR}}$  the following
\begin{lemma}\label{Moll-Thm}
Let  \(p \in \mathcal{P}^{\log}(\mathbb{R}^n)\) and let 
\(\psi \in L^1(\mathbb{R}^n)\). Assume that the least bell-shaped majorant 
\(\Psi\) of \(\psi\) is integrable. Then
\[
\| f * \psi_\varepsilon \|_{L^{p(\cdot)}(\mathbb{R}^n)} 
\;\le\; c \, \|\Psi\|_{1} \, \| f \|_{L^{p(\cdot)}(\mathbb{R}^n)}
\]
for all \(f \in L^{p(\cdot)}(\mathbb{R}^n)\), where $\psi_\varepsilon(x):=\varepsilon^{-n}\psi(\varepsilon^{-1}x),\, x\in \mathbb{R}^n$ and $c$ is a constant of   depends only on \(p\) and  \(n\).
\end{lemma}

Define for  $x\in \mathbb{R}^{n}$, $t>0$ and $m>0$ the function 
\begin{equation}\label{Not-eta}
\eta _{t,m}(x):=t^{-n}(1+t^{-1}\left\vert x\right\vert )^{-m}.
\end{equation}%
 Note that $\eta _{t,m}\in L^{1}(\mathbb{R}^{n})$ when $m>n$ and that $\big\|\eta _{t,m}\big\|_{1}=c(m)$ is independent
of $t$. Also, for all $t>0$ and $m>n$ the least bell-shaped majorant of
$\eta _{t,m}$ is  $\eta _{t,m}$ itself , then by Lemma~\ref{Moll-Thm} if $p\in \mathcal{P}^{\text{log}}(\mathbb{R}^{n})$
\begin{equation} 
\Vert \eta _{t,m} \ast f\Vert _{L^{p(\cdot )}(\mathbb{R}^{n})}\leq c \Vert f\Vert _{L^{p(\cdot )}(\mathbb{R}^{n})} \label{eneq p norm}
\end{equation}
 are independent of $t>0$.  We refer to the recent monographs \cite{CF,DHHR} for further
properties, historical remarks and references on variable exponent Lebesgue spaces.

We say that the variable Lebesgue space $L^{p(\cdot)}(\mathbb{R}^n)$ is continuously embedded into $L^{q(\cdot)}(\mathbb{R}^n)$, and write
$
L^{p(\cdot)}(\mathbb{R}^n) \hookrightarrow L^{q(\cdot)}(\mathbb{R}^n),
$
if
$
L^{p(\cdot)}(\mathbb{R}^n) \subset L^{q(\cdot)}(\mathbb{R}^n)
$
and there exists a constant $C>0$ such that
\[
\|f\|_{L^{q(\cdot)}(\mathbb{R}^n)} \le C \|f\|_{L^{p(\cdot)}(\mathbb{R}^n)}
\]
for all  f $\in L^{p(\cdot)}(\mathbb{R}^n)$. 
Now, recall some results on embeddings, the following two  lemmas are \cite[Theorem 3.3.1]{DHHR} and \cite[Proposition 4.1.8]{DHHR}, respectively.
\begin{lemma}\label{lem1}
Let \(p, q \in \mathcal{P}(\mathbb{R}^n)\) where $p\leq q$ almost everywhere. Define the exponent \(s \in \mathcal{P}(\mathbb{R}^n)\) by
\[
\frac{1}{s(x)} := \max \left\{\, \frac{1}{p(x)} -\frac{1}{q(x)} ,\, 0 \, \right\}, \qquad \forall x \in A.
\]
If  \(1 \in L^{s(\cdot)}(\mathbb{R}^n)\), then
\[
L^{q(\cdot)}(\mathbb{R}^n) \hookrightarrow L^{p(\cdot)}(\mathbb{R}^n)
\]
with embedding norm at most \(2 \, \|1\|_{L^{s(\cdot)}(\mathbb{R}^n)}\).
\end{lemma}
Note that, in a bounded domain $\Omega \subset \mathbb{R}^n$ with $|\Omega| < \infty$, then by \cite[Lemma 3.2.12]{DHHR},
\begin{equation}\label{1 in bound in P}
\|1\|_{L^{s(\cdot)}(\Omega)} \leq C \max\{ |\Omega|^{\frac{1}{s^-}} , |\Omega|^{\frac{1}{s^+}} \},
\end{equation}
and by \cite[Corollary 3.3.4.]{DHHR} the variable Lebesgue spaces satisfy the continuous embedding $L^{q(\cdot)}(\Omega) \hookrightarrow L^{p(\cdot)}(\Omega)$ provided that $p(x) \leq q(x)$ for almost every $x \in \Omega$.

Unlike in the case of classical Lebesgue spaces, this embedding is possible even though  $|\mathbb{R}^n| = \infty$. 
The condition  \(1 \in L^{s(\cdot)}(\mathbb{R}^n)\) where \(s \in \mathcal{P}(\mathbb{R}^n)\) is of great importance in establishing many embedding results and plays an important role in convolution estimates in variable Lebesgue spaces. As an example, if $s(x)=1+\sum_{i=1}^n|x_i|, \,x \in \mathbb{R}^n$ then for $\lambda > 1 $ we have 
$$\varrho _{s(\cdot )}\left(\frac{1}{\lambda} \right)= \int_{\mathbb{R}^n}\,\left(\frac{1}{\lambda}\right)^{s(x)} \,dx=
\int_{\mathbb{R}^n}\,\lambda^{-\left(1+\sum_{i=1}^n|x_i|\right)} \,dx=\frac{2^n}{\lambda\log^n(\lambda)} ,
$$
and $\varrho _{s(\cdot )}\left(\frac{1}{\lambda} \right)=\infty$ if $0 <\lambda \leq 1$.  Thus \(1 \in L^{s(\cdot)}(\mathbb{R}^n)\) and 
\begin{equation*}
\left\Vert 1\right\Vert _{L^{s(\cdot)}(\mathbb{R}^n)}=\inf \Big\{\lambda >1:2^n \leq \lambda\log^n(\lambda)\Big\} \leq e^2 .
\end{equation*}%

\begin{lemma} \label{lem2}
Let \(p, q \in \mathcal{P}^{\log}(\mathbb{R}^n)\) with \(p_\infty = q_\infty\).  
If \(s \in \mathcal{P}(\mathbb{R}^n)\) is given by
\[
\frac{1}{s} := \left| \frac{1}{p} - \frac{1}{q} \right|,
\]
then \(1 \in L^{s(\cdot)}(\mathbb{R}^n)\), and
\begin{equation*}
L^{\max\{p(\cdot),q(\cdot)\}}(\mathbb{R}^n) 
\;\hookrightarrow\; L^{q(\cdot)}(\mathbb{R}^n) 
\;\hookrightarrow\; L^{\min\{p(\cdot),q(\cdot)\}}(\mathbb{R}^n).
\end{equation*}
\end{lemma}

In particular if  \(p, q \in \mathcal{P}^{\log}(\mathbb{R}^n)\) with $ p\leq q $ and  $p_\infty = q_\infty$, then by      
Lemma~\ref{lem2} we have  \(1 \in L^{s(\cdot)}(\mathbb{R}^n)\) and 
\begin{equation*}
L^{q(\cdot)}(\mathbb{R}^n) \hookrightarrow L^{p(\cdot)}(\mathbb{R}^n).
\end{equation*}

If $p\leq p_\infty $ that is $p^+=p_\infty$, by taking $q=p_\infty$ then, again by Lemma~\ref{lem2}
\begin{equation}\label{Embed-Eq}
L^{p_\infty}(\mathbb{R}^n) \;\hookrightarrow\; L^{p(\cdot)}(\mathbb{R}^n).
\end{equation} 
and \(1 \in L^{s(\cdot)}(\mathbb{R}^n)\) where $\frac{1}{s} :=  \frac{1}{p} - \frac{1}{p_\infty}$.

\begin{definition}\label{def:RieszPotential}
Let $0 < \beta < n$. For a measurable function $f$, the \textit{Riesz potential operator} $\mathcal{R}_{\beta}(f) : \mathbb{R}^n \to [0, +\infty]$ is defined by
\begin{equation*} \label{eq:RieszPotential}
\mathcal{R}_{\beta}(f)(x) := \int_{\mathbb{R}^n} \frac{|f(y)|}{|x - y|^{n-\beta}} \, dy.
\end{equation*}
\end{definition}
The Riesz potential operator is an essential tool in the study of Partial Differential Equations (PDEs). Its boundedness between Lebesgue space  known as the Hardy-Littlewood-Sobolev inequality has been widely applied to the Navier–Stokes equations to establish existence and regularity results.
 
The following theorem establishes the boundedness of the Riesz potential operator within the framework of variable Lebesgue spaces, providing the necessary conditions on the exponent $p(\cdot)$ to ensure its continuity, 
a proof of this theorem can be consulted in \cite[Section 6.1]{DHHR}. 
\begin{theorem} \label{Thm:RieszBoundedness}
Let $p(\cdot) \in \mathcal{P}_{\log}(\mathbb{R}^n)$ and $0 < \beta < n/p^+$. Then, there exists a constant $C > 0$ such that
\begin{equation*} \label{eq:RieszEstimate}
\| \mathcal{R}_{\beta}(f) \|_{L^{q(\cdot)}} \le C \| f \|_{L^{p(\cdot)}}, 
\end{equation*}
where the exponent $q(\cdot)$ satisfies the relation
\begin{equation*}
\frac{1}{q(\cdot)} = \frac{1}{p(\cdot)} - \frac{\beta}{n}.
\end{equation*}
\end{theorem}

\begin{theorem}[Banach--Picard principle] \label{Thm:BanachPicard}
Consider a Banach space $(E,\|\cdot\|_E)$ and a bounded bilinear application  $B : E \times E \to E$ 
such that
\[
\|B(u,u)\|_E \le C_B \|u|_E \|u\|_E .
\]
Given $e_0 \in E$ such that $\|e_0\|_E \le \delta$ with   $0 < \delta < \frac{1}{4C_B},$ 
the equation
$
u = e_0 - B(u,u)
$
admits a unique solution $u \in E$ which satisfies 
$
\|u\|_E \le 2\delta.
$
\end{theorem}

\bigskip
\section{On Convolution in Variable Lebesgue Spaces}\label{sec-conv}

This section is devoted to results on convolutions in variable Lebesgue spaces.
We establish estimates of the form
\[
\|k * f\|_Z \le C \|k\|_X \|f\|_Y,
\]
where $X$, $Y$, and $Z$ are variable Lebesgue spaces satisfying appropriate structural relations.
We also develop several auxiliary lemmas that will be used in the proofs of existence theorems
for the  Fractional Navier--Stokes equations in the next section. Recall that, in the classical setting, if $1 \le p,q,r \le \infty$ satisfy 
$
\frac{1}{p} + \frac{1}{q} = 1 + \frac{1}{r},
$ 
then  
$
\|k * f\|_{L^r(\mathbb{R}^n)} \le c \|k\|_{L^p(\mathbb{R}^n)} \|f\|_{L^q(\mathbb{R}^n)}
$ . The following theorem, \cite[Theorem 2.7]{Samak1996}, extends this result to the framework of variable Lebesgue spaces, where the exponents $p$ and $q$ may vary pointwise, while the exponent $r$ remains constant.

\begin{theorem}\label{Thm-Lr-Conv}
Let $p,q\in \mathcal{P}(\mathbb{R}^n)$  and \(r\ge 1\) with $p^+,q^+<+\infty$, such that  almost every \(x\in\mathbb{R}^n\)
\begin{equation}\label{eq:exponent-relation}
\frac{1}{p(x)}+\frac{1}{q(x)} = 1 + \frac{1}{r}.
\end{equation}

 Suppose $k \in L^{\,q^-}(\mathbb{R}^n)\cap L^{\,q^+}(\mathbb{R}^n).$ Then  there
exists a constant \(c>0\) (depending on $p$ and \(r\)) such that
\begin{equation*}
\|k * f\|_{L^r(\mathbb{R}^n)} \le c C \,\|f\|_{L^{p(\cdot)}(\mathbb{R}^n)},
\end{equation*}
for all $ f\in L^{p(\cdot)}(\mathbb{R}^n)$,  where
\begin{align*}
C=A^{\nu}(\|k\|_{L^{q^-}}^{\frac{q^-}{r}} + \|k\|_{L^{q^+}}^{\frac{q^+}{r}}),\ A = \|k\|_{L^{q^-}} + \|k\|_{L^{q^+}} ,\, \text{ and }\,
 \nu =
\begin{cases}
1 - \frac{q^+}{r}, & \text{if } A \le 1,\\
1 - \frac{q^-}{r}, & \text{if } A > 1.
\end{cases}
\end{align*} 
\end{theorem}

We may see that we have $C \leq 2 A^{\eta}$ where 
\begin{align}
\eta =
\begin{cases}\label{eta-label}
1 - \frac{q^+-q^-}{r}, & \text{if } A \le 1,\\
1 + \frac{q^+-q^-}{r}, & \text{if } A > 1.
\end{cases}
\end{align}

The following two lemmas are established to facilitate the proof of the existence results for the fractional incompressible Navier-–Stokes equations (\ref{NS_Intro}), which are detailed in Section \ref{Sec-Exis}. They concern  estimates for the convolution 
$\eta_{t,m} * f$ 
between the variable Lebesgue spaces $L^{p(\cdot)}(\mathbb{R}^n)$ and
$L^{\frac{p(\cdot)}{2}}(\mathbb{R}^n)$.

\begin{lemma}\label{LemmaCom 1} Let \(p \in \mathcal{P}^{\log}(\mathbb{R}^n)\) and $m>n$ with $2\leq p^-\leq p^+ <+\infty$ and  $p^+=p_\infty$, then 
\begin{align*}
\|\eta_{t,m}*f\|_{L^{p(\cdot)}(\mathbb{R}^n)}
\leq C t^{-n\vartheta(t)} \|f\|_{L^{ \frac{p(\cdot)}{2}}(\mathbb{R}^n)}
\end{align*}
for every $t\in (0,+\infty)$, where
\begin{align}
\vartheta(t) =
\begin{cases}\label{rho-label}
\left(\frac{2}{p^-}-\frac{1}{p_\infty} \right)
 \left(\frac{2+p_\infty(1-\frac{2}{p^-})}{1+p_\infty(1-\frac{2}{p^-})}- \frac{1}{p_\infty-1}\right) , & \text{if } 0<t \le 1,\\
\frac{1}{p_\infty}
\left(\frac{p_\infty(1-\frac{2}{p^-})}{1+p_\infty(1-\frac{2}{p^-})}+ \frac{1}{p_\infty-1}\right)
 , & \text{if } t > 1,
\end{cases}
\end{align}
and  the constant $C$ depends on the exponent $p$ and the constant $m$.
\end{lemma}

\begin{proof}
Applying Theorem  \ref{Thm-Lr-Conv} with $\frac{p}{2}$,  $r=p_\infty=p^+$  and $q(\cdot)\in \mathcal{P}(\mathbb{R}^n)$ defined be the relation $\frac{1}{p(x)/2}+\frac{1}{q(x)}=1+\frac{1}{p_\infty}$, hence
$$
\|\eta_{t,m} * f\|_{L^{p_\infty}(\mathbb{R}^n)} \le c \left( \|\eta_{t,m}\|_{L^{q^-}(\mathbb{R}^n)} + \|\eta_{t,m}\|_{L^{q^+}(\mathbb{R}^n)} \right)^{\eta} \,\|f\|_{L^{\frac{p(\cdot)}{2}}(\mathbb{R}^n)}\ ,
$$
for every  $t >0$, where $\eta$ is from (\ref{eta-label}) and 
 $$\frac{1}{q^-}=1-\frac{1}{p_\infty} \ , \  \frac{1}{q^+}=1+\frac{1}{p_\infty}-\frac{2}{p^-}.$$
For every $t\in (0,+\infty)$ we have 
\begin{equation*}
\|\eta_{t,m}\|_{L^{q^-}(\mathbb{R}^n)} = t^{-n(1-\frac{1}{q^-})} \|\eta_{t,m{q^-}}\|_{L^1(\mathbb{R}^n)}^{\frac{1}{q^-}} \ ,\ \|\eta_{t,m}\|_{L^{q^+}(\mathbb{R}^n)} = t^{-n(1-\frac{1}{q^+})} \|\eta_{t,m{q^+}}\|_{L^1(\mathbb{R}^n)}^{\frac{1}{q^+}}.
\end{equation*}
It follows 
\begin{align*}
\left( \|\eta_{t,m}\|_{L^{q^-}} + \|\eta_{t,m}\|_{L^{q^+}} \right)^{\eta} &\leq M^{1 + \frac{q^+-q^-}{r}} t^{-n(1-\frac{1}{q^+})(1 + \frac{q^+-q^-}{r})},\, \text{ if } 0< t\leq 1 \\
\left( \|\eta_{t,m}\|_{L^{q^-}} + \|\eta_{t,m}\|_{L^{q^+}} \right)^{\eta} & \leq M^{1 + \frac{q^+-q^-}{r}} t^{-n(1-\frac{1}{q^-})(1 - \frac{q^+ - q^-}{r})},\, \text{ if }  t> 1.
\end{align*}
where $M=\max\{1\, ,\, \|\eta_{t,m{q^-}}\|_{L^1}^{\frac{1}{q^-}}+\|\eta_{t,m{q^+}}\|_{L^1}^{\frac{1}{q^+}}\,\}$. Since \(p \in \mathcal{P}^{\log}(\mathbb{R}^n)\) and  $p^+=p_\infty$ by Lemma \ref{lem2} we have  $L^{p_\infty}(\mathbb{R}^n) 
\;\hookrightarrow\; L^{p(\cdot)}(\mathbb{R}^n)$, it follows   that 
\begin{align*}
\|\eta_{t,m}*f\|_{L^{p(\cdot)}(\mathbb{R}^n)}\leq C \|\eta_{t,m}*f\|_{L^{p_\infty}}
\leq C t^{-n\vartheta(t)} \|f\|_{L^{ \frac{p(\cdot)}{2}}(\mathbb{R}^n)}
\end{align*}
where 
\[
\vartheta(t) =
\begin{cases}
\left(1-\frac{1}{q^+} \right) \left( 1 + \frac{q^+-q^-}{p_\infty}\right) , & \text{if } 0<t \le 1,\\
\left(1-\frac{1}{q^-} \right)\left(  1 - \frac{q^+-q^-}{p_\infty}\right) , & \text{if } t > 1.
\end{cases}
\]
To write $\vartheta$ in terms of $p_\infty$ and $p^-$,  we can see that 
\begin{align*}
1-\frac{1}{q^-} = \frac{1}{p_\infty},\ 1-\frac{1}{q^+}=\frac{2}{p^-}-\frac{1}{p_\infty},\
\frac{q^-}{p_\infty}=\frac{1}{p_\infty-1} ,\ \frac{q^+}{p_\infty}=\frac{1}{1+p_\infty(1-\frac{2}{p^-})},
\end{align*} 
which yields the expression of $\vartheta$ presented in (\ref{rho-label}). The proof is completed.

\end{proof}
\begin{remark}\label{Remark on mq}
We observe from the proof that the conclusion of Lemma~\ref{LemmaCom 1} is guaranteed by the integrability conditions $\eta_{t,mq^{\pm}} \in L^{1}(\mathbb{R}^n)$. This condition is satisfied whenever $mq^-=m\frac{p_\infty}{p_\infty-1} > n$. We also have 
 $\vartheta(t) \geq \frac{2}{p^-}-\frac{1}{p_\infty} $ for every $0< t\leq 1$ and  $\vartheta(t) \leq \frac{n}{p_\infty} $ for every $t> 1$, Furthermore, it follows from \eqref{eta-label} that the parameter $\vartheta$ satisfies the inequality, $\vartheta|_{(1,+\infty)} \leq  \vartheta|_{(0,1]}$, that is the value of $\vartheta$ on $(1,+\infty)$ is less than or equal its value on $(0,1]$. 
 In the case of a constant exponent, we have $\vartheta(t) = \frac{1}{p}$ for every $t > 0 $, which coincides with the scaling in the classical Young's convolution inequality.  A more general version of  Lemma~\ref{LemmaCom 1} is stated at the end of Section~\ref{Sec-Exis} in Theorem~\ref{THm - eta Lr-Lp conv}.
\end{remark}
\begin{lemma}\label{LemmaCom 2}
Let \(p \in \mathcal{P}^{\log}(\mathbb{R}^n)\) and $m>n$ with $2\leq p^- $ and $p_\infty=\infty$ , then 
\begin{align*}
\|\eta_{t,m}*f\|_{L^{p(\cdot)}(\mathbb{R}^n)}
\leq C t^{-n\vartheta(t)} \|f\|_{L^{ \frac{p(\cdot)}{2}}(\mathbb{R}^n)}
\end{align*}
where $\vartheta(t)=\frac{2}{p^-}$ if $0<t \le 1$ and $\vartheta(t)=0$ if $t>1$.
\end{lemma}

\begin{proof}Since \(p \in \mathcal{P}^{\log}(\mathbb{R}^n)\) and  $p_\infty=\infty$,  by Lemma \ref{lem2} we have  $L^{\infty}(\mathbb{R}^n) 
\;\hookrightarrow\; L^{ p(\cdot)}(\mathbb{R}^n)$,  H\"{o}lder's inequality yields 
\begin{align*}
\|\eta_{t,m}*f\|_{L^{p(\cdot)}(\mathbb{R}^n)}&\leq C \|\eta_{t,m}*f\|_{L^{\infty}(\mathbb{R}^n)}\\
&\leq C \,\underset{x\in \mathbb{R}^{n}}{\text{ess-sup}} \|\eta_{t,m}(x-\cdot)\|_{L^{ q(\cdot)}(\mathbb{R}^n)}\|f\|_{L^{ \frac{p(\cdot)}{2}}(\mathbb{R}^n)}
\end{align*}
where $q$ is the conjugate exponent of $\frac{p}{2}$. Now, take  $\lambda > c_m t^{-n\vartheta(t)}$ arbitrary, where $c_m=\max\{\, 1,\|\eta _{t,m}\|_{L^{1}(\mathbb{R}^n)}\,\}$, then for all $x\in\mathbb{R}^n$ and $t>1$
\begin{align*}
\varrho _{q(\cdot)}\left(\frac{\eta _{t,m}(x-\cdot)}{\lambda}   \right)
\leq \int_{\mathbb{R}^n} \frac{\eta _{t,m}(x-y) }{c_m} \,dy
\le 1.
\end{align*}
and for all $x\in\mathbb{R}^n$ and $0<t \le 1$
\begin{align*}
\varrho _{q(\cdot)}\left(\frac{\eta _{t,m}(x-\cdot)}{\lambda}  \right)
&= \int_{\mathbb{R}^n}\left(\frac{t^{-n\left(1-\frac{1}{q(y)} \right)}}{\lambda}\right)^{q(y)} \frac{t^{-n}}{(1+t^{-1}|x-y|)^m} \,dy\\
&\leq \int_{\mathbb{R}^n} \left(\frac{t^{-\frac{2n}{p^-}}}{\lambda}\right)^{q(y)} \eta _{t,m}(x-y)\,dy\\
&\le 1.
\end{align*}
by letting $\lambda\to c_m t^{-n\vartheta(t)}$, it follows that $\underset{x\in \mathbb{R}^{n}}{\text{ess-sup}} \|\eta_{t,m}(x-\cdot)\|_{L^{ q(\cdot)}(\mathbb{R}^n)}\leq c_m t^{-n\vartheta(t)}$ for every $t>0$, thus the proof is completed.

\end{proof}
\begin{remark}
 Lemmas  \ref{LemmaCom 1} and \ref{LemmaCom 2} still hold if  $1\in L^{s(\cdot)}(\mathbb{R}^{n})$ where the exponent $s(\cdot)$ is defined by
\[
\frac{1}{s(x)} := \max \left\{\, \frac{1}{p(x)} -\frac{1}{p^+} ,\, 0 \, \right\}, \qquad \forall x \in \mathbb{R}^{n}.
\]
since $L^{p^+}(\mathbb{R}^n) \; \hookrightarrow \; L^{ p(\cdot)}(\mathbb{R}^n)$  by Lemma \ref{lem1}.
\end{remark}

The following theorem generalize the  convolution theorem on constant variable Lebesgue spaces 
\begin{theorem}\label{THM 1}  Let $p,q,r\in[1,\infty]$ and $A,B,C,D \in \mathcal{P}(\mathbb{R}^n)$ such that  $\frac{1}{p}+\frac{1}{q}+\frac{1}{r}=1$ and $\frac{A(x)}{p}+\frac{C(x)}{r} = 1,\frac{B(x)}{q}+\frac{D(x)}{r} = 1$. If 
$k\in L^{A(\cdot)}(\mathbb{R}^n)\cap L^{C(\cdot)}(\mathbb{R}^n) $ and $f\in L^{B(\cdot)}(\mathbb{R}^n)\cap L^{D(\cdot)}(\mathbb{R}^n)$,  then 
\begin{equation*}
\|k*f\|_{L^r(\mathbb{R}^n)}\leq \|k\|_{L^{A(\cdot)}(\mathbb{R}^n)\cap L^{C(\cdot)}(\mathbb{R}^n)}\|f\|_{L^{B(\cdot)}(\mathbb{R}^n)\cap L^{D(\cdot)}(\mathbb{R}^n)}.
\end{equation*}
\end{theorem}

\begin{proof}
For $p,q,r\in(1,\infty)$, define two functions $\alpha,\beta$ such that $A(x)=p\alpha(x)$ and $B(x)=q\beta(x)$ then $C(x)=r(1-\alpha(x),D(x)=r(1-\beta(x))$ for every $x\in\mathbb{R}^n$.
Consider arbitrary two numbers   $\lambda_1>\max(\|k\|_{L^{A(\cdot)}(\mathbb{R}^n)},\|k\|_{ L^{C(\cdot)}(\mathbb{R}^n)})$ and $\lambda_2>\max(\|f\|_{L^{B(\cdot)}(\mathbb{R}^n)},\|f\|_{ L^{D(\cdot)}}(\mathbb{R}^n))$, hence 
\begin{equation*}
\varrho_{A(\cdot)}\left(\frac{k}{\lambda_1}\right) \leq 1 , \varrho_{C(\cdot)}\left(\frac{k}{\lambda_1}\right)\leq 1 , 
\varrho_{B(\cdot)}\left(\frac{f}{\lambda_2}\right)\leq 1 , \varrho_{D(\cdot)}\left(\frac{f}{\lambda_2}\right)\leq 1.
\end{equation*}
Set $\mu=\lambda_1\lambda_2$. For all  $x,t\in\mathbb{R}^n$ we have 
\begin{align*}
\frac{|k(x-t)f(t)|}{\mu}=\left|\frac{k(x-t)}{\lambda_1}\right|^{\alpha(x-t)}
\left|\frac{f(t)}{\lambda_2}\right|^{\beta(t)}
\left|\frac{k(x-t)}{\lambda_1}\right|^{1-\alpha(x-t)}
\left|\frac{f(t)}{\lambda_2}\right|^{1-\beta(t)},
\end{align*}
thus  applying  H\"{o}lder's inequality yields,
\begin{align*}
\int_{\mathbb{R}^n} \frac{|k(x-t)f(t)|}{\mu} \,dt&=\int_{\mathbb{R}^n}  \left|\frac{k(x-t)}{\lambda_1}\right|^{\alpha(x-t)}
\left|\frac{f(t)}{\lambda_2}\right|^{\beta(t)}
\left|\frac{k(x-t)}{\lambda_1}\right|^{1-\alpha(x-t)}
\left|\frac{f(t)}{\lambda_2}\right|^{1-\beta(t)}\,dt\\
&\leq \Big( \int_{\mathbb{R}^n}  \left|\frac{k(x-t)}{\lambda_1}\right|^{p\alpha(x-t)}\,dt \Big)^{\frac{1}{p}}
\Big( \int_{\mathbb{R}^n}  \left|\frac{f(t)}{\lambda_2}\right|^{q\beta(t)}\Big)^{\frac{1}{q}}\\
&\hspace{3cm}\times \Big(\int_{\mathbb{R}^n}  \left|\frac{k(x-t)}{\lambda_1}\right|^{r(1-\alpha(x-t))}
\left|\frac{f(t)}{\lambda_2}\right|^{r(1-\beta(t))}\,dt\Big)^{\frac{1}{r}},
\end{align*}
regarding the first two terms on the left-hand side, we observe that,
\begin{align*}
\int_{\mathbb{R}^n}  \left|\frac{k(x-t)}{\lambda_1}\right|^{p\alpha(x-t)}\,dt =\varrho_{A(\cdot)}\left(\frac{k}{\lambda_1}\right) \leq 1 \ \text{ and } \ 
\int_{\mathbb{R}^n}  \left|\frac{f(t)}{\lambda_2}\right|^{q\beta(t)} =\varrho_{B(\cdot)}\left(\frac{f}{\lambda_2}\right)\leq 1.
\end{align*}
Consequently, the following inequality holds for almost every $x \in \mathbb{R}^n$,
\begin{align*}
\int_{\mathbb{R}^n} \frac{|k(x-t)f(t)|}{\mu} \,dt\leq \left(\int_{\mathbb{R}^n}  \left|\frac{k(x-t)}{\lambda_1}\right|^{C(x-t)}
\left|\frac{f(t)}{\lambda_2}\right|^{D(t)}\,dt\right)^{\frac{1}{r}},
\end{align*}
which is equivalent to saying that for all $x \in \mathbb{R}^n$,
\begin{align*}
\left|\frac{k*f(x)}{\mu}\right|^{r}\leq \int_{\mathbb{R}^n}  \left|\frac{k(x-t)}{\lambda_1}\right|^{C(x-t)}
\left|\frac{f(t)}{\lambda_2}\right|^{D(t)}\,dt.
\end{align*}
Now, integrating both sides over $\mathbb{R}^n$ yields the following estimate,  
\begin{align*}
\int_{\mathbb{R}^n}  \left|\frac{k*f(x)}{\mu}\right|^{r}\,dx\leq  
\int_{\mathbb{R}^n}  \left|\frac{f(t)}{\lambda_2}\right|^{D(t)}\,\int_{\mathbb{R}^n}  \left|\frac{k(x-t)}{\lambda_1}\right|^{C(x-t)} \,dx\,dt\leq 1. 
\end{align*}
This leads to the following inequality  $\|k*f\|_{L^r(\mathbb{R}^n)}\leq \lambda_1\lambda_2$, by letting 
\begin{align*}
\lambda_1\to\max(\|k\|_{L^{A(\cdot)}(\mathbb{R}^n)},\|k\|_{ L^{C(\cdot)}(\mathbb{R}^n)})  \ \text{ and  } \
\lambda_2\to\max(\|f\|_{L^{B(\cdot)}(\mathbb{R}^n)},\|f\|_{ L^{D(\cdot)}(\mathbb{R}^n)}),
\end{align*}
 we deduce that 
\begin{align*}
\|k*f\|_{L^r(\mathbb{R}^n)}&\leq \max(\|k\|_{L^{A(\cdot)}(\mathbb{R}^n)},\|k\|_{ L^{C(\cdot)}(\mathbb{R}^n)}) \max(\|f\|_{L^{B(\cdot)}(\mathbb{R}^n)},\|f\|_{ L^{D(\cdot)}(\mathbb{R}^n)})\\
&=
\|k\|_{L^{A(\cdot)}(\mathbb{R}^n)\cap L^{C(\cdot)}(\mathbb{R}^n)}\|f\|_{L^{B(\cdot)}(\mathbb{R}^n)\cap L^{D(\cdot)}(\mathbb{R}^n)}.
\end{align*}
If $r=\infty$ then $\frac{1}{p}+\frac{1}{q}=1, \, p,q\in [1,\infty]$, $ A=p$ and $B=q$, by H\"{o}lder's inequality, 
\begin{equation*}
\|k*f\|_{L^\infty(\mathbb{R}^n)}\leq \|k\|_{L^{p}(\mathbb{R}^n)}\|f\|_{ L^{q}(\mathbb{R}^n)}\leq 
 \|k\|_{L^{p}(\mathbb{R}^n) \cap L^{C(\cdot)}(\mathbb{R}^n) }\|f\|_{L^{q}(\mathbb{R}^n)\cap L^{ D(\cdot)}(\mathbb{R}^n) },
\end{equation*}
if $r=1$ then $p=q=\infty$ and  $C=D=1$, by Young's inequality, $\|k*f\|_{L^1(\mathbb{R}^n)}\leq \|k\|_{L^{1}(\mathbb{R}^n)} \|f\|_{ L^{ 1}(\mathbb{R}^n)} $ $ \leq\|k\|_{ L^{A(\cdot)}(\mathbb{R}^n)\cap L^{1}(\mathbb{R}^n)  }\|f\|_{ L^{ B(\cdot)}(\mathbb{R}^n)\cap L^{1}(\mathbb{R}^n) }$. The proof is complete.

\end{proof}
\begin{remark}We provide the following observations concerning Theorem \ref{THM 1} supplemented by an example to illustrate the applicability of our results.
\begin{enumerate}
\item  Primarily, under the assumption that the exponents fulfill the conditions of Theorem \ref{THM 1}, then $A,B,C,D$ are bounded, also this result coincide with the known Young's inequality, to see this, for real numbers $p,q,r\geq 1  $ such that $ \frac{1}{p}+\frac{1}{q}=1+\frac{1}{r}$ then $ \frac{1}{p'}+\frac{1}{q'}+\frac{1}{r}=1$,
$ \frac{p}{q'}+\frac{p}{r}=1$ and $ \frac{q}{p'}+\frac{q}{r}=1$, thus taking $A=C=p$ and $B=D=q$  by Theorem \ref{THM 1}
\begin{align*}
\|k*f\|_{L^r(\mathbb{R}^n)}\leq \|k\|_{L^{p}(\mathbb{R}^n)}\|f\|_{L^{q}(\mathbb{R}^n)},
\end{align*}
which is Young's inequality. 
\item Let $A,B\in \mathcal{P}(\mathbb{R}^n)$ be bounded  and $p,q,r\in[1,\infty)$
such that  $p>A^+,q>B^+$, $\frac{1}{p}+\frac{1}{q}<1$ and $\frac{1}{p}+\frac{1}{q}+\frac{1}{r}=1$. Define
$C,D\in \mathcal{P}_0(\mathbb{R}^n)$ by  $C(x) = r(1-\frac{A(x)}{p}),D(x) =r( 1-\frac{B(x)}{q})$ then by Theorem \ref{THM 1}
\begin{equation*}
\|k*f\|_{L^r(\mathbb{R}^n)}\leq 
\|k\|_{L^{A(\cdot)}(\mathbb{R}^n)\cap L^{r\big(1-\frac{A(\cdot)}{p}\big)}(\mathbb{R}^n)}\|f\|_{L^{B(\cdot)}(\mathbb{R}^n)\cap L^{r\big( 1-\frac{B(\cdot)}{q}\big)}(\mathbb{R}^n)}.
\end{equation*}
\item If $ p=q=r=3$ and $A(x)=1+|\sin(x)|,C(x)=2-|\sin(x)|,B(x)=1+e^{-|x|},D(x)=2-e^{-|x|}$, then 
\begin{equation*}
\|k*f\|_{L^3(\mathbb{R})}\leq \|k\|_{L^{1+|\sin(x)|}(\mathbb{R}^n)\cap L^{2-|\sin(x)|}(\mathbb{R})}\|f\|_{L^{1+e^{-|x|}(\mathbb{R}^n)}\cap L^{2-e^{-|x|}}(\mathbb{R})},
\end{equation*}
for all  $k\in L^{1+|\sin(x)|}(\mathbb{R})\cap L^{2-|\sin(x)|}(\mathbb{R}) ,f\in L^{1+e^{-|x|}}(\mathbb{R})\cap L^{2-e^{-|x|}}(\mathbb{R})$.
\end{enumerate}
\end{remark}

By virtue of the embedding from \eqref{Embed-Eq}, we have  $L^{r_\infty}(\mathbb{R}^n) \;\hookrightarrow\; L^{r(\cdot)}(\mathbb{R}^n)$ if \(r \in \mathcal{P}^{\log}(\mathbb{R}^n)\)  $r_\infty = r^+$, then we have the following corollary.
\begin{corollary}\label{coro - of conv}
Let \(r \in \mathcal{P}^{\log}(\mathbb{R}^n)\) with  $r_\infty = r^+$, \(A,B,C,D \in \mathcal{P}(\mathbb{R}^n)\), $p,q\in[1,\infty]$ such that $\frac{1}{p}+\frac{1}{q}+\frac{1}{r^+}=1,\frac{A(x)}{p}+\frac{C(x)}{r^+} = 1,\frac{B(x)}{q}+\frac{D(x)}{r^+} = 1$. Then  
\begin{equation*}
\|k*f\|_{L^{r(\cdot)}(\mathbb{R}^n)}\leq c\|k\|_{L^{A(\cdot)}(\mathbb{R}^n)\cap L^{C(\cdot)}(\mathbb{R}^n)}\|f\|_{L^{B(\cdot)}(\mathbb{R}^n)\cap L^{D(\cdot)}(\mathbb{R}^n)}
\end{equation*}
for all  $k\in L^{A(\cdot)}(\mathbb{R}^n)\cap L^{C(\cdot)} (\mathbb{R}^n),f\in L^{B(\cdot)}(\mathbb{R}^n)\cap L^{D(\cdot)}(\mathbb{R}^n)$. 
\end{corollary}

We observe that in Theorem~\ref{THM 1} the exponent \(r\) is constant. Most convolution
results have difficulty accommodating a variable exponent \(r(\cdot)\) without
imposing additional assumptions, as illustrated in Corollary~\ref{coro - of conv}.
In the following theorems, we  provide a framework that allows
for variable exponents \(r(\cdot)\).

\begin{theorem}\label{THM1-2} Let $p,q\in(1,\infty)$ and $r,A,B,C,D \in \mathcal{P}_0(\mathbb{R}^n)$ such that $1<r^-<r^+<+\infty$, $\frac{1}{p}+\frac{1}{q}=1$ and $\frac{A(x)}{p(r')^-}+\frac{C(x)}{r^-} = 1,\frac{B(x)}{q(r')^-}+\frac{D(x)}{r^-} = 1$. If 
$k\in L^{A(\cdot)}(\mathbb{R}^n)\cap L^{C(\cdot)}(\mathbb{R}^n)\cap L^\infty(\mathbb{R}^n) $ and $f\in L^{B(\cdot)}(\mathbb{R}^n)\cap L^{D(\cdot)}(\mathbb{R}^n)\cap L^\infty(\mathbb{R}^n)$,  then 
\begin{equation*}
\|k*f\|_{L^{r(\cdot)}(\mathbb{R}^n)}\leq \|k\|_{L^{A(\cdot)}(\mathbb{R}^n)\cap L^{C(\cdot)}(\mathbb{R}^n)\cap L^\infty(\mathbb{R}^n)}\|f\|_{L^{B(\cdot)}(\mathbb{R}^n)\cap L^{D(\cdot)}(\mathbb{R}^n)\cap L^\infty(\mathbb{R}^n)}.
\end{equation*}
\end{theorem}

\begin{proof} The proof follows an argument similar to that of Theorem~\ref{THM 1}, we therefore omit the details. Define tow functions $\alpha,\beta$ such that $A(x)=p(r')^-\alpha(x)$ and $B(x)=q(r')^-\beta(x)$ then $C(x)=r^-(1-\alpha(x),D(x)=r^-(1-\beta(x))$ for all $x\in\mathbb{R}^n$. Let $\mu:=\lambda_1\lambda_2$ where  $\lambda_1>\max(\|k\|_{L^{A(\cdot)}(\mathbb{R}^n)},\|k\|_{ L^{C(\cdot)}(\mathbb{R}^n)},\|k\|_{ L^{\infty}(\mathbb{R}^n)})$, $\lambda_2>\max(\|f\|_{L^{B(\cdot)}(\mathbb{R}^n)},\|f\|_{ L^{D(\cdot)}(\mathbb{R}^n)},\|f\|_{ L^{\infty}(\mathbb{R}^n)})$. Since $\frac{1}{p}+\frac{1}{q}=1$, it follows that 
$\frac{1}{r(x)}+\frac{1}{pr'(x)}+\frac{1}{qr'(x)}=1$ for every  $x\in\mathbb{R}^n$, then   H\"{o}lder's inequality yields
\begin{align*}
\int_{\mathbb{R}^n} \frac{|k(x-t)f(t)|}{\mu} \,dt &=\int_{\mathbb{R}^n} \left|\frac{k(x-t)}{\lambda_1}\right|^{\alpha(x-t)}
\left|\frac{f(t)}{\lambda_2}\right|^{\beta(t)}
\left|\frac{k(x-t)}{\lambda_1}\right|^{1-\alpha(x-t)}
\left|\frac{f(t)}{\lambda_2}\right|^{1-\beta(t)}\,dt\\
&\leq \left( \int_{\mathbb{R}^n} \left|\frac{k(x-t)}{\lambda_1}\right|^{p(r')^-\alpha(x-t)}\,dt \right)^{\frac{1}{pr'(x)} }
\left( \int_{\mathbb{R}^n} \left|\frac{f(t)}{\lambda_2}\right|^{q(r')^-\beta(t)}\,dt\right)^{\frac{1}{qr'(x)}}\\
&\hspace{3cm}\, \times \,\left(\int_{\mathbb{R}^n} \left|\frac{k(x-t)}{\lambda_1}\right|^{r(x)(1-\alpha(x-t))}
\left|\frac{f(t)}{\lambda_2}\right|^{r(x)(1-\beta(t))}\,dt\right)^{\frac{1}{r(x)}}\\
&\leq \left(\int_{\mathbb{R}^n} \left|\frac{k(x-t)}{\lambda_1}\right|^{C(x-t)}
\left|\frac{f(t)}{\lambda_2}\right|^{D(t)}\,dt\right)^{\frac{1}{r(x)}}.
\end{align*}
for every $x\in\mathbb{R}^n$. Therefore
\begin{align*}
\int_{\mathbb{R}^n} \left|\frac{k*f(x)}{\mu}\right|^{r(x)}\leq \int_{\mathbb{R}^n} \int_{\mathbb{R}^n} \left|\frac{k(x-t)}{\lambda_1}\right|^{C(x-t)}
\left|\frac{f(t)}{\lambda_2}\right|^{D(t)}\,dt\,dx \leq 1.
\end{align*}
By choosing $\lambda_1$ and $\lambda_2$ sufficiently small, the desired result follows, thereby completing the proof.

\end{proof}

\begin{theorem}\label{Prop2}
Let \(r,p \in \mathcal{P}(\mathbb{R}^n)\) and $q$ is the conjugate exponent of $p$, then we have
\begin{enumerate}
\item $ \|k*f\|_{L^{r(\cdot)}(\mathbb{R}^n)}\leq C  \|k\|_{L^{\infty}(\mathbb{R}^n)\cap L^{r^-}(\mathbb{R}^n)} \|f\|_{L^{p(\cdot)}(\mathbb{R}^n)\cap L^{1}(\mathbb{R}^n)} ,$
\item There exists a positive constant $C$ such that for every $\tau>0$ 
$$ \|\eta _{\tau,m}*f\|_{L^{r(\cdot)}(\mathbb{R}^n)}\leq C  \tau^{-n \vartheta(\tau)}  \|f\|_{L^{p(\cdot)}(\mathbb{R}^n)\cap L^{1}(\mathbb{R}^n)} ,$$
where 
\begin{small}
\begin{equation}
\vartheta(\tau)=
\begin{cases}
\frac{r^-}{r^+}(1-\frac{1}{r^-})+ \zeta(\tau)(1-\frac{r^-}{r^+}) , & \text{if }\ 0<\tau\leq 1 \land 1-\frac{1}{r^-} \leq \zeta(\tau) \ \text{ or }\ \tau > 1 \land 1-\frac{1}{r^-} > \zeta(\tau)\ ; \\
1-\frac{1}{r^-} , & \text{if }\ 0<\tau\leq 1 \land 1-\frac{1}{r^-} > \zeta(\tau) \ \text{ or }\ \tau > 1 \land 1-\frac{1}{r^-} \leq  \zeta(\tau),
\end{cases}
\end{equation}
\end{small}
if $r^-<\infty$ and $\vartheta(\tau)=\zeta(\tau)$ if $r^-=\infty$ with $\zeta(\tau)=\frac{1}{p^-}$ if $\tau \in (0,1]$ and $\zeta(\tau)=\frac{1}{p^+}$ if $\tau>1$.
\item For every $1\leq \nu \leq r^-$ there exists a positive constant $C$ such that for every $\tau>0$ 
\begin{align*}
\|\eta _{\tau,m}*f\|_{L^{r(\cdot)}(\mathbb{R}^n)}\leq C  \tau^{-n  \omega(\tau)}  \|f\|_{L^{p(\cdot)}(\mathbb{R}^n)\cap L^{\nu}(\mathbb{R}^n)} ,
\end{align*}
where
\begin{equation}\label{Omega-General}
\omega(\tau)=
\begin{cases} \frac{1}{p^-}(1-\frac{\nu}{r^+}) , & \text{if } 0<\tau\leq 1  \\
\frac{1}{p^+}(1-\frac{\nu}{r^-}) , & \text{if }  \tau > 1 
\end{cases}
\end{equation}
with the convention that $\frac{\nu}{\infty}=0$ for $\nu\in [1,+\infty]$.
\item  If \(1 \in L^{s(\cdot)}(\mathbb{R}^n)\) where the exponent \(s \in \mathcal{P}(\mathbb{R}^n)\)is defined by
\[
\frac{1}{s(x)} := \max \left\{\, \frac{1}{p^-} -\frac{1}{p(x)} ,\, 0 \,\right\}, \qquad \forall x \in \mathbb{R}^n,
\]
  then  
$$ \|k*f\|_{L^{r(\cdot)}(\mathbb{R}^n)}\leq C  \|k\|_{L^{q^+}(\mathbb{R}^n)\cap L^{r^-}(\mathbb{R}^n)} \|f\|_{L^{p(\cdot)}(\mathbb{R}^n)\cap L^{1}(\mathbb{R}^n)}, $$
\end{enumerate}
\end{theorem}

\begin{proof}
We assume $r^- < \infty$, the case $r^- = \infty$ implies that $r(\cdot)$ is constant and equal to $\infty$, the proofs in this scenario are considerably simpler, following arguments analogous to the first case, and are therefore omitted. 
 Let   $\mu:=\lambda_1\lambda_2$ where   $\lambda_1$ and $\lambda_2$ will be chosen separately for each assertion. To prove the first assertion, consider $\lambda_1 > \|f\|_{L^{p(\cdot)}(\mathbb{R}^n)\cap L^{1}(\mathbb{R}^n)}$    and $\lambda_2 > \|k\|_{L^{+\infty}(\mathbb{R}^n)\cap L^{r^-}(\mathbb{R}^n)}$, then we have 
\begin{align*}
\int_{\mathbb{R}^n} \frac{|k(x-t)f(t)|}{\mu}\,dt &\leq \frac{ \left\|f(t)\right\|_{L^{1}} }{\lambda_1}\leq 1
\end{align*}
hence, for every $x\in\mathbb{R}^n$
\begin{align*}
 \left|\frac{ k*f(x)}{\mu} \right|^{r(x)}\leq  \ \left( \int_{\mathbb{R}^n} \frac{|f(x-t)k(t)|}{\mu} \,dt\right)^{r^-},
\end{align*}
therefore, by Minkowski inequality
\begin{align*}
\int_{\mathbb{R}^n}  \left|\frac{ k*f(x)}{\mu} \right|^{r(x)}\,dx &\leq \int_{\mathbb{R}^n}  \left( \int_{\mathbb{R}^n} \frac{|k(x-t)f(t)|}{\mu} \,dt\right)^{r^-}\,dx \\
&\leq  \left(\int_{\mathbb{R}^n}  \left(\int_{\mathbb{R}^n}\left| \frac{k(x-t)f(t)}{\mu}\right|^{r^-} \,dx \right)^{1/r^-}\,dt \right)^{r^-}\\
&\leq 1.
\end{align*} 
Therefore, for the first assertion we get $\varrho _{r(\cdot )}\left(\frac{ k*f}{\mu} \right)\leq 1$. For the second assertion of the theorem we first take  
$\beta = \tau^{-n\zeta(\tau)}$ where $\zeta(\tau)=\frac{1}{p^-}$ if $\tau\in (0,1]$, $\zeta(\tau)=\frac{1}{p^+}$ if $\tau>1$ and  $c_m:=1+\|\eta _{\tau,m}\|_{L^{1}(\mathbb{R}^n)}=1+\|\eta _{1,m}\|_{L^{1}(\mathbb{R}^n)}$, then 
\begin{align*}
\varrho _{q(\cdot)}\left(\frac{\eta _{\tau,m}(x-\cdot)}{c_m\beta}   \right)
&= \int_{\mathbb{R}^n} \left|\frac{\eta _{\tau,m}(x-y)}{c_m\beta}  \right|^{q(y)} \,dy\\
&\leq \int_{\mathbb{R}^n} \left(\frac{\tau^{-n\vartheta(\tau)}}{c_m\beta}\right)^{q(y)} \eta _{\tau,m}(x-y)\,dy\\
&\le 1.
\end{align*}
thus, by H\"{o}lder's inequality  there exists $c_1\ge 1$ such that 
\begin{align*}
\int_{\mathbb{R}^n} \frac{|\eta _{\tau,m}(x-t)f(t)|}{c_m\beta\lambda_1}\,dt &\leq c \left\|\frac{\eta _{\tau,m}(x-\cdot)}{c_m\beta}  \right\|_{L^{q(\cdot)}(\mathbb{R}^n)} \left\|\frac{f(\cdot)}{\lambda_1}  \right\|_{L^{p(\cdot)}(\mathbb{R}^n)} \leq c_1,
\end{align*}
As shown above, applying Minkowski's integral inequality again yields
\begin{align}
\varrho _{r(\cdot )}\left(\frac{ \eta _{\tau,m}*f}{C\beta\lambda_1} \right)^{1/r^-} &\leq   \, \frac{\|\eta _{\tau,m}\|_{L^{r^-}(\mathbb{R}^n)}\|f\|_{L^1(\mathbb{R}^n)}}{C\beta\lambda_1} \label{MinK-App-Inequ}\\
&\leq  \, \frac{c \tau^{-n(1-\frac{1}{r^-})}\|f\|_{L^1(\mathbb{R}^n)}}{C\tau^{-n\vartheta(\tau)}\lambda_1} \nonumber\\
&\leq \tau^{-n(1-\frac{1}{r^-}-\zeta(\tau))} \nonumber
\end{align}
for every $\tau >0$ and for some positive constant $C$. Then we have the inequality
\begin{align*}
\tau^{nr^-(1-\frac{1}{r^-}-\zeta(\tau))}\varrho _{r(\cdot )}\left(\frac{ \eta _{\tau,m}*f}{C\mu} \right) \leq 1\ ,
\end{align*}
for every $\tau >0$. Therefore if $0<\tau\leq 1$ and $1-\frac{1}{r^-}-\zeta(\tau)\leq 0$ or $\tau\geq 1$ and $1-\frac{1}{r^-}-\zeta(\tau)\geq 0$ one has 
\begin{align*}
\varrho _{r(\cdot )}\left(\frac{ \tau^{nr^-(1-\frac{1}{r^-}-\zeta(\tau))\frac{1}{r^+}}\eta _{\tau,m}*f}{C\mu} \right) \leq 1.
\end{align*}
And if $0<\tau\leq 1$ and  $1-\frac{1}{r^-}-\zeta(\tau)\geq 0$ or $\tau\geq 1$ and  $1-\frac{1}{r^-}-\zeta(\tau)\leq 0$ we have 
\begin{align*}
\varrho _{r(\cdot )}\left(\frac{ \tau^{nr^-(1-\frac{1}{r^-}-\zeta(\tau))\frac{1}{r^-}}\eta _{\tau,m}*f}{C\mu} \right) \leq 1.
\end{align*}
Now, define 
\[
\lambda_2 := 
\begin{cases}
\tau^{-\frac{nr^-}{r^+}(1-\frac{1}{r^-})-n \zeta(\tau)(1-\frac{r^-}{r^+})} , & \text{if }\ 0<\tau\leq 1 \land 1-\frac{1}{r^-} \leq \zeta(\tau) \ \text{ or }\ \tau > 1 \land 1-\frac{1}{r^-} > \zeta(\tau)\ ; \\
\tau^{-n(1-\frac{1}{r^-})} , & \text{if }\ 0<\tau\leq 1 \land 1-\frac{1}{r^-} > \zeta(\tau) \ \text{ or }\ \tau > 1 \land 1-\frac{1}{r^-} \leq  \zeta(\tau).
\end{cases}
\]
We conclude that  for every $\tau >0$
\begin{align*}
\varrho _{r(\cdot )}\left(\frac{ \eta _{\tau,m}*f}{C\lambda_1\lambda_2} \right) \leq 1
\end{align*}
Similarly, the third assertion is proved by taking $\lambda_1 > \|f\|_{L^{p(\cdot)}(\mathbb{R}^n)\cap L^{\nu}(\mathbb{R}^n)}$ and again    by Minkowski's integral inequality and  replacing $r^-$ with $\nu$ and interchanging the roles of $\eta_{\tau,m}$ and $f$ in (\ref{MinK-App-Inequ}),  we have 
$$
\varrho _{r(\cdot )}\left(\frac{ \eta _{\tau,m}*f}{C\beta\lambda_1} \right)^{1/\nu} \leq   \, \frac{\|\eta _{\tau,m}\|_{L^1(\mathbb{R}^n)}\|f\|_{L^{\nu}(\mathbb{R}^n)}}{C\beta\lambda_1}\leq \tau^{n\zeta(\tau)}.$$
By applying arguments analogous to those used above, we conclude that 
$\varrho _{r(\cdot )}\left(\frac{ \eta _{\tau,m}*f}{C \lambda_1\lambda_2} \right) \leq 1$ for every $\tau >0$ where $\lambda_2:=\tau^{-n  \omega(\tau)}$ where $\omega(\tau)$ is given in (\ref{Omega-General}).   For the fourth assertion,  we may see that  
$\frac{1}{s(x)} = \max \left\{\, \frac{1}{q(x)} -\frac{1}{q^+} ,\, 0 \,\right\}, \ \forall x \in \mathbb{R}^n,$
since \(1 \in L^{s(\cdot)}(\mathbb{R}^n)\) then for arbitrary $\lambda_2 > \|k\|_{L^{q^+}(\mathbb{R}^n)\cap L^{r^-}(\mathbb{R}^n)}$ and  $\lambda_1 > \|f\|_{L^{p(\cdot)}(\mathbb{R}^n)\cap L^{1}(\mathbb{R}^n)}$ it follows from Lemma \ref{lem1} that
\begin{align*}
\left\|\frac{k(x-\cdot)}{\lambda_2}  \right\|_{L^{q(\cdot)}(\mathbb{R}^n)}
\leq C \left\|\frac{k}{\lambda_2}  \right\|_{L^{q^+}(\mathbb{R}^n)}\leq c,
\end{align*}
which implies by similar arguments as above  that  $\varrho _{r(\cdot )}\left(\frac{ k*f}{C\lambda_1\lambda_2} \right)\leq 1$ for some positive constant $C$. In each assertion there exists $C>0$ such that 
\begin{align*}
\left\|k*f  \right\|_{L^{r(\cdot)}(\mathbb{R}^n)}\leq C \lambda_1\lambda_2\ , \ \left\|\eta _{\tau,m}*f  \right\|_{L^{r(\cdot)}(\mathbb{R}^n)}\leq C \lambda_1\lambda_2,
 \end{align*}
where $\lambda_1$ and $\lambda_2$ are presented above for each assertion.  Letting $\lambda_1 \to \|f\|_{L^{p(\cdot)}(\mathbb{R}^n)\cap L^{1}(\mathbb{R}^n)}$, $\lambda_1 \to \|f\|_{L^{p(\cdot)}(\mathbb{R}^n)\cap L^{\nu}(\mathbb{R}^n)}$  and 
$\lambda_2 \to\|k\|_{ L^{\infty}(\mathbb{R}^n)\cap L^{r^-}(\mathbb{R}^n)}$, 
or $\lambda_2 \to \|k\|_{L^{q^+}(\mathbb{R}^n)\cap L^{r^-}(\mathbb{R}^n)}$.   
By combining the results of the individual assertions, we conclude that all claims of the theorem hold. Thus, the proof is complete.

\end{proof}

\begin{lemma}\label{Lemma of LV cap Lp}
Let \(p \in \mathcal{P}(\mathbb{R}^n)\) and $\nu$ a constant exponent with $2\leq  p^-$, $2\leq \nu \leq  2p^-$ and let $m>n$. 
Then  there exists a positive constant $C(m,p)$ such that for every $t>0$ and $u,v\in L^{p(\cdot)}(\mathbb{R}^n)\cap L^{\nu}(\mathbb{R}^n)$,
\begin{align}\label{Omega-Lp-0}
\|\eta_{t,m}*(uv)\|_{L^{p(\cdot)}\cap L^{\nu}(\mathbb{R}^n)}&\leq C \max(t^{-n \omega(t)},t^{-\frac{n}{\nu}}) \|u\|_{L^{p(\cdot)}(\mathbb{R}^n)\cap L^{\nu}(\mathbb{R}^n)}
\|v\|_{L^{p(\cdot)}(\mathbb{R}^n)\cap L^{\nu}(\mathbb{R}^n)}.
\end{align}
where
\begin{equation}\label{Omega-Lp}
\omega(t)=
\begin{cases} \frac{2}{p^-}(1-\frac{\nu}{2p^+}) , & \text{if } 0<t \leq 1  \ ;\\
\frac{2}{p^+}(1-\frac{\nu}{2p^-}) , & \text{if }  t > 1 \ . 
\end{cases}
\end{equation}
Moreover, if \(p \in \mathcal{P}^{\log}(\mathbb{R}^n)\) with $2\leq p^- = p_\infty $ then 
\begin{align}
\|\eta_{t,m}*(uv)\|_{L^{p(\cdot)}}&\leq C t^{-n \phi(t)} \|u\|_{L^{p(\cdot)}(\mathbb{R}^n)}
\|v\|_{L^{p(\cdot)}(\mathbb{R}^n)}.\label{Omega-Lp-Log}
\end{align}
where $\phi $ is defined by
\begin{equation}\label{Omega-Lp2}
\phi(t)=
\begin{cases} \frac{2}{p^-}-\frac{1}{p^+} \ , & \text{if } 0<t \leq 1  \ ;\\
\frac{1}{p^+} \  , & \text{if }  t > 1 \ . 
\end{cases}
\end{equation}
\end{lemma}
\begin{proof}
By Theorem~\ref{Prop2}(3) and H\"{o}lder's inequality, for $\frac{p(\cdot)}{2} $   we have 
\begin{align}
\|\eta _{t,m}*(uv)\|_{L^{p(\cdot)}(\mathbb{R}^n)} &\leq C  t^{-n  \omega(t)}  \|uv\|_{L^{\frac{ p(\cdot)}{2}}(\mathbb{R}^n)\cap L^{\frac{\nu}{2}}(\mathbb{R}^n)} ,\nonumber \\
&\leq C t^{-n  \omega(t)}  \|u\|_{L^{ p(\cdot)}(\mathbb{R}^n)\cap L^{\nu}(\mathbb{R}^n)}\|u\|_{L^{ p(\cdot)}(\mathbb{R}^n)\cap L^{\nu}(\mathbb{R}^n)} , \label{Ene-1-Omega-Lp2}
\end{align}
the function $\omega(t)$ is calculated according to (\ref{Omega-General}) and is explicitly given in (\ref{Omega-Lp}). 
Applying Young's inequality  with the relation $1+\frac{1}{\nu}=\frac{1}{\nu'}+\frac{1}{\nu/2}$ and H\"{o}lder's inequality we obtain
\begin{align*}
\|\eta_{t,m}*(uv)\|_{L^{\nu}(\mathbb{R}^n)} &\leq C\  \|\eta_{t,m}\|_{L^{\nu'}(\mathbb{R}^n)} \|\, uv\,\|_{ L^{\nu/2}(\mathbb{R}^n)} \\
& \leq C\ \|\eta_{t,m}\|_{L^{\nu'}(\mathbb{R}^n)} \| u\|_{ L^{\nu}(\mathbb{R}^n)} \| v\|_{ L^{\nu}(\mathbb{R}^n)}\\
&\leq C\  t^{\frac{-n}{\nu}} \|u\|_{L^{p(\cdot)}(\mathbb{R}^n)\cap L^{\nu}(\mathbb{R}^n)}
\|v\|_{L^{p(\cdot)}(\mathbb{R}^n)\cap L^{\nu}(\mathbb{R}^n)}
\end{align*}
for every $t>0$ , where we have used the inequality $\|\eta_{t,m}\|_{L^{\nu'}(\mathbb{R}^n)} \leq c t^{\frac{-n}{\nu}}$. 
If \(p \in \mathcal{P}^{\log}(\mathbb{R}^n)\) with $2\leq p^- = p_\infty $ then by Lemma \ref{lem2} we have  $ L^{ p(\cdot)}(\mathbb{R}^n) \;\hookrightarrow\;  L^{p^-}(\mathbb{R}^n)$, therefore, from inequality (\ref{Ene-1-Omega-Lp2}) by taking $\nu=p^-$ we conclude~(\ref{Omega-Lp-Log}) and $\phi(t)$ is calculated from $\omega(t)$.  Thus, the proof is complete.

\end{proof}
\begin{remark}\qquad
\begin{enumerate}
\item With similar arguments as in the proof of (\ref{Omega-Lp-Log}) we can see that  if \(p \in \mathcal{P}^{\log}(\mathbb{R}^n)\) with $2\leq p^- = p_\infty $ then 
\begin{align}\label{P-p/2Comp-2}
\|\eta_{t,m}*f\|_{L^{p(\cdot)}(\mathbb{R}^n)}
\leq C t^{-n\phi(t)} \|f\|_{L^{ \frac{p(\cdot)}{2}}(\mathbb{R}^n)}
\end{align}
where $\phi$ is given in (\ref{Omega-Lp2}). Consequently, we derive two distinct versions of the previous estimate  correspond to cases $p^- = p_\infty$ and $p^+=p_\infty$, with the latter specifically detailed in Lemmas \ref{LemmaCom 1} and~\ref{LemmaCom 2}.
\item  When the exponents $p=1$ and $r=2$ in  Theorem~\ref{Prop2}(2) or $\nu=p\leq r$ for constant exponents in  Theorem~\ref{Prop2}(3), the results reduces to the standard estimate derived from the classical Young's convolution inequality.
\item The parameters $\omega(t)$ and $\varphi(t)$ satisfy the inequalities $\omega|_{(1,\infty)} \leq \omega|_{(0,1]}$ and $\varphi|_{(1,\infty)} \leq \varphi|_{(0,1]}$, where both reduce to equalities in the case of a constant exponent $p$.
\end{enumerate}
\end{remark}
Next, we  apply Theorem~\ref{Prop2}(2) to derive an estimates similar to (\ref{Omega-Lp-0}) and (\ref{Omega-Lp-Log}) if \(p \in \mathcal{P}^{\log}(\mathbb{R}^n)\) with $2 = p^- = p_\infty $. 
\begin{lemma}\label{Lemma of L2 cap Lp} 
Let \(p \in \mathcal{P}(\mathbb{R}^n)\) with $2\leq p^-$ and $m>n$. 
Then , there exists a positive constant $C(m,p)$ such that for every $t>0$ and $u,v\in L^{p(\cdot)}(\mathbb{R}^n)\cap L^{2}(\mathbb{R}^n)$,
\begin{align*}
\|\eta_{t,m}*(uv)\|_{L^{p(\cdot)}\cap L^{2}(\mathbb{R}^n)}&\leq C \max(t^{-n \varsigma (t)},t^{-\frac{n}{2}}) \|u\|_{L^{p(\cdot)}(\mathbb{R}^n)\cap L^{2}(\mathbb{R}^n)}
\|v\|_{L^{p(\cdot)}(\mathbb{R}^n)\cap L^{2}(\mathbb{R}^n)},
\end{align*}
where
\begin{equation}\label{rho of Lp L2}
\varsigma(t)=
\begin{cases}
\frac{p^-}{p^+}(1-\frac{1}{p^-})+ \zeta(t)(1-\frac{p^-}{p^+}) , & \text{if }\ 0<t \leq 1 \land\ 2\leq  p^- \leq 3 \ \text{ or }\ t > 1 \land 1-\frac{1}{p^-} > \frac{2}{p^+} \ ; \\
1-\frac{1}{p^-} , & \text{if }\ 0<t \leq 1 \land p^- > 3 \ \text{ or }\ t > 1 \land 1-\frac{1}{p^-} \leq  \frac{2}{p^+} .
\end{cases}
\end{equation}
Moreover, if \(p \in \mathcal{P}^{\log}(\mathbb{R}^n)\) with $ p^- = p_\infty =2$ then 
\begin{align*}
\|\eta_{t,m}*(uv)\|_{L^{p(\cdot)}}&\leq C t^{-n K(t)} \|u\|_{L^{p(\cdot)}(\mathbb{R}^n)}
\|v\|_{L^{p(\cdot)}(\mathbb{R}^n)},
\end{align*}
where 
\begin{equation}\label{varphi of Lp L2}
K(t)=
\begin{cases}
\frac{1}{p^+}+ \zeta(t)(1-\frac{2}{p^+}) , & \text{ if }\ 0<t \leq 1 \ \text{ or }\ t > 1 \land  p^+> 4\ ; \\
\ \frac{1}{2} \ , & \text{if } t > 1 \land p^+ \leq  4 .
\end{cases}
\end{equation}
 is calculated by replacing $p^-=2$ in the expression of $\varsigma (t)$ and $\zeta(t)=\frac{2}{p^-}$ if $ 0<t \leq 1$ and $\zeta(t)=\frac{2}{p^+}$ if $t>1$. 
\end{lemma}

\begin{remark}\quad \label{Remar - L^2 }
\begin{enumerate}
\item Since $\frac{n}{2} \geq 1$ for $n=2, 3$, the estimates in Lemma~\ref{Lemma of L2 cap Lp} lacks the necessary flexibility, see Remark \ref{Rem 4.4}. Despite this, in some cases the function $\varsigma$  may be finer  than $\omega$ presented in (\ref{Omega-Lp}). For example  $K(t)\geq  \omega(t) $ for every  $t > 1$, which means that $K $ is finer than $\omega$ on $(1;+\infty)$.

\item In all cases, the function $K$ from (\ref{varphi of Lp L2}) satisfies the estimate $K|_{(1,\infty)} \leq K|_{(0,1]}$.
\end{enumerate}
\end{remark}

\bigskip
\section{Applications to Fractional Navier–-Stokes Equations}\label{Sec-Exis}
By Duhamel's formula, the fractional incompressible Navier–-Stokes system \eqref{NS_Intro} can be expressed in its mild (integral) formulation as follows,
\begin{equation}\label{Duham1}
u(t,x)=
e^{-t(-\Delta)^{\alpha}} u_0(x)
+
\int_{0}^t
e^{-(t-s)(-\Delta)^{\alpha}} f(s, x)
ds
-
\int_{0}^{t} e^{-(t-s)(-\Delta)^{\alpha}} \mathbb{P}\operatorname{div}\bigl(u \otimes u\bigr)(s)\,ds,
\end{equation}
where  \( e^{-t(-\Delta)^{\alpha}} \) denotes the semigroup generated by the fractional Laplacian $(-\Delta)^{\alpha}$, and  $\mathbb{P} = I - \nabla(\Delta)^{-1}\operatorname{div}$ is  the Leray projector. Let $g_{\alpha,t}=\mathcal{F}^{-1}\!\bigl[e^{-t|\xi|^{2\alpha}}\bigr]$ be the usual fractional heat kernel and $K_{\alpha,t}$ the kernel of $ e^{-t(-\Delta)^{\alpha}} \mathbb{P}\operatorname{div}(\cdot)$ then we may write Equation~(\ref{Duham1}) as
\begin{equation}\label{NS_Mild}
u(t,x)=
g_{\alpha,t}\ast u_0(x)
+
\int_{0}^t
g_{\alpha,t-s}\ast f(s, x)
ds
-
\int_{0}^t
K_{\alpha,t-s}\ast (u \otimes u)(s, x)
ds.
\end{equation}
We introduce the bilinear integral operator
\begin{equation}\label{def-B}
B(u,u)(x,t) =  -
\int_{0}^t
K_{\alpha,t-s}\ast (u \otimes u)(s, x)
ds,
\end{equation}
so that Equation~(\ref{NS_Mild}) can be rewritten as
\begin{equation*}
u(x,t) = e_0 + B(u,u)(x,t), 
\end{equation*}
with 
\begin{align}
e_0(t,x)= g_{\alpha,t}\ast u_0(x)
+
\int_{0}^t
g_{\alpha,t-s}\ast f(s, x)\ 
ds.
\end{align}

The \(j\)-th component of $B(u,u)$ can be written as
\begin{align}\label{eqB=F}
B_j(u,u) =- \int_{0}^{t} \sum_{h,k =1}^3 (K^\alpha_{t-s})_{j} * (u_h u_k)(s)\,ds
=- \sum_{h,k=1}^3 \int_{0}^{t} K^{j;h,k}_{\alpha,t-s} * (u_h u_k)(s)\,ds.
\end{align}
where $\mathcal{F}( K^{j;h,k}_{\alpha,t})(\xi) = i\xi_k \Bigl( \delta_{jh} - \frac{\xi_j \xi_h}{|\xi|^2} \Bigr) e^{-t|\xi|^{2\alpha}}$.  For a more detailed treatment of these kernels and additional references, we refer the reader to \cite{QW2023}.

We now state several fundamental pointwise estimates for the the kernels $g_{\alpha,t}$ and  $K_{\alpha,t}$. We begin with the following lemma \cite[Lemma 2.1]{Miao2008}.
\begin{lemma} For $\alpha > 0$, the kernel function \( g_\alpha(x)=\mathcal{F}^{-1}\!\bigl[e^{-|\xi|^{2\alpha}}\bigr](x) \) has the following pointwise estimate
\[
|g_{\alpha}(x)| \leq C(1 + |x|)^{-n - 2\alpha}, \, \ x\in \mathbb{R}^n\ .
\]
\end{lemma}

Now, for all $t>0, x\in\mathbb{R}^n$ we have 
\begin{equation}
g_{\alpha,t}(x)=t^{\frac{-n}{2\alpha}}\,g_\alpha \,\bigl(t^{ \frac{-1}{2\alpha}}x\bigr),\label{scal-g}
\end{equation}
thus, following the notation in \eqref{Not-eta}, it follows that there exists a constant $C > 0$ such that
\begin{align}\label{est-g_eta}
|g_\alpha(x)| \leq C \ \eta _{1,n + 2\alpha}(x)\ \text{ and } \ |g_{\alpha,t}(x)| \leq C \  \eta _{t^{\frac{1}{2\alpha}},n + 2\alpha}(x),
\end{align}
for all $t>0, x\in\mathbb{R}^n$. 
We have similar estimate for  the kernel \( K_{\alpha,t}(x) \), by the estimate (2.9) from \cite{QW2023}, there exists a constant $C > 0$ such that
\begin{equation*}
|K_{\alpha,t}(x)| \leq C\,\bigl(t^{\frac{1}{2\alpha}} + |x|\bigr)^{-n-1},
\end{equation*} 
for all $t>0, x\in\mathbb{R}^n$, following the notation in \eqref{Not-eta}, we have
\begin{equation}\label{eneq-F}
|K_{\alpha,t}(x)| \leq C t^{-\frac{1}{2\alpha}}
 \eta_{t^{\frac{1}{2\alpha}},n+1}(x)\, , \, \forall\ t>0, x\in\mathbb{R}^n.
\end{equation}
Thus, using the above estimate, we can avoid the difficulty arising from the fact that the Riesz transform 
is not continuous on $ L^{\infty}(\mathbb{R}^n) $.

For $\kappa>0$, the \( \kappa\)-th derivative of the kernel $ g_\alpha(x) $ of $ g_{\alpha,t}(x) $ is defined by
\begin{equation*}
g_\alpha^\kappa(x) = (-\Delta)^{\kappa/2} g_\alpha(x), \quad 
g_{\alpha,t}^\kappa(x) = (-\Delta)^{\kappa/2} g_{\alpha, t}(x).
\end{equation*}
By \cite[Lemma 2.2]{Miao2008}, for \( \kappa > 0 \) we have  the following pointwise estimate,
\begin{equation*}
|g_\alpha^\kappa(x)| \leq C(1 + |x|)^{-n - \kappa}=C\,\eta _{1,n + \kappa}(x), \quad x \in \mathbb{R}^n,
\end{equation*}

Thanks to the scaling (\ref{scal-g}), the kernel function \( g_{\alpha,t}^\kappa \) satisfies,
\begin{equation*}
g_{\alpha,t}^\kappa(x)
= t^{-\frac{\kappa}{2\alpha}}\, t^{-\frac{n}{2\alpha}}
  g_\alpha^\kappa(x)\!\left( \frac{x}{t^{\frac{1}{2\alpha}}} \right).
\end{equation*}
Then, for $\kappa >0$ it follows that
\begin{equation}\label{Est-g-al}
\left| g_{\alpha,t}^\kappa(x) \right| \leq C t^{-\frac{\kappa}{2\alpha}}\eta _{t^{\frac{1}{2\alpha}},n + \kappa}(x), \quad \forall x \in \mathbb{R}^n.
\end{equation}
\subsection{Estimates for the Fractional Heat Semigroup} \label{Est - HEAT semigroup}

We first recall the following classical space–-time estimates for the fractional heat semigroup, as established in \cite[Lemma 3.1]{Miao2008}.
\begin{theorem}\label{Lemma - MIAO}
Let $1\leq r \leq p \leq +\infty $. Then, for $\alpha,\kappa >0$,   there exists $C>0$ such that 
\begin{enumerate}
    \item  $\left\|
    g_{\alpha,t }*
    \varphi\right\|_{L^p}
\leq
C t^{-\frac{n}{2 \alpha}\left(\frac{1}{r}-\frac{1}{p}\right)}\|\varphi\|_{L^r}$,
\item $\left\|(-\Delta)^{\kappa / 2} g_{\alpha,t } * 
\varphi\right\|_{L^p} \leq C t^{-\frac{\kappa}{2 \alpha}-\frac{n}{2 \alpha}\left(\frac{1}{r}-\frac{1}{p}\right)}\|\varphi\|_{L^r}$.
\end{enumerate}
\end{theorem}
As an initial application of the estimates derived in Section \ref{sec-conv}, we now establish the following generalization theorems for variable Lebesgue spaces.
\begin{theorem}\label{Prop - 1 - app} 
Let $p\in \mathcal{P}^{\log}(\mathbb{R}^n)$ and $r\in \mathcal{P}(\mathbb{R}^n)$ such that the exponents satisfy either $r^+\leq p^+=p_\infty<\infty$  or $p^+=\infty$
Then 
\begin{enumerate}
    \item  $\left\|
    g_{\alpha,t }*
    \varphi\right\|_{L^{p(\cdot)}(\mathbb{R}^n)}
\leq
C t^{\frac{-n\sigma(t)}{2 \alpha}}\|\varphi\|_{L^{r(\cdot)}(\mathbb{R}^n)}$,
\item $\left\|(-\Delta)^{\kappa / 2} g_{\alpha,t } * 
\varphi\right\|_{L^{p(\cdot)}(\mathbb{R}^n)} \leq C t^{-\frac{\kappa}{2 \alpha}-\frac{n\sigma(t)}{2 \alpha}}\|\varphi\|_{L^{r(\cdot)}(\mathbb{R}^n)}$,
\end{enumerate}
where 
\begin{equation}\label{def-of-sig(t)}
\sigma(t) =
\begin{cases}
 \left( \frac{1}{r^-} - \frac{1}{p_\infty} \right) \left( 1 +\frac{1}{p_\infty}  \frac{\frac{1}{r^-}-\frac{1}{r^+}}{\left( 1+\frac{1}{p_\infty}-\frac{1}{r^-} \right) \left( 1+\frac{1}{p_\infty}-\frac{1}{r^+} \right)} \right) , & \text{if } 0<t \le 1,\\
\left( \frac{1}{r^+} - \frac{1}{p_\infty} \right)\left(  1 - \frac{1}{p_\infty}\frac{\frac{1}{r^-}-\frac{1}{r^+}}{\left( 1+\frac{1}{p_\infty}-\frac{1}{r^-} \right) \left( 1+\frac{1}{p_\infty}-\frac{1}{r^+} \right)} \right) , & \text{if } t > 1.
\end{cases}
\end{equation}
\end{theorem}
\begin{proof} Define $q\in \mathcal{P}(\mathbb{R}^n)$ by the relation
\begin{equation*}
\frac{1}{r(x)}+\frac{1}{q(x)} = 1 + \frac{1}{p_\infty} .
\end{equation*}
Using inequality \eqref{est-g_eta} and proceeding with calculations analogous to those in Lemma \ref{LemmaCom 1} for the bounded case $p^+ < \infty$, we have
\begin{align*}
\left\| g_{\alpha,t }*\varphi\right\|_{L^{p_\infty}}
 \leq C \left\|\,  \eta _{t^{\frac{1}{2\alpha}},n + 2\alpha} *|\,   \varphi \, | \right\|_{L^{p_\infty}}
 \leq C t^{-\frac{n\sigma(t)}{2 \alpha}}\|\varphi\|_{L^{r(\cdot)}}
\end{align*}
Since \(p \in \mathcal{P}^{\log}(\mathbb{R}^n)\) and  $p^+=p_\infty$ by Lemma \ref{lem2} we have  $L^{p_\infty}(\mathbb{R}^n) 
\;\hookrightarrow\; L^{p(\cdot)}(\mathbb{R}^n)$. Having established the result for $p^+ < \infty$, we now address the endpoint case $p^+ = \infty$. By Lemma \ref{lem2}, we have the continuous embedding $L^{\infty} \hookrightarrow L^{p(\cdot)}$. Following arguments similar to those in Lemma \ref{LemmaCom 2}, the result follows immediately.
  The second part is obtained similarly, utilizing the estimate \eqref{Est-g-al}. To determine the value of $\sigma(t)$, we consider the following relationships,
\begin{align*}
\frac{1}{q^-} =1+\frac{1}{p_\infty}-\frac{1}{r^+} \  , \  \frac{1}{q^+}=1+\frac{1}{p_\infty}-\frac{1}{r^-}\  \text{ and } \
q^+-q^- = \frac{\frac{1}{r^-}-\frac{1}{r^+}}{\left( 1+\frac{1}{p_\infty}-\frac{1}{r^-} \right) \left( 1+\frac{1}{p_\infty}-\frac{1}{r^+} \right)}\ \cdot
\end{align*}
 Furthermore, the expression for $\sigma(t)$ presented in (\ref{def-of-sig(t)}) remains well-defined in the endpoint cases where $p^+ = \infty$ or $r^+=\infty$ under the standard convention that $\frac{1}{\infty} = 0$. Thus, the proof is complete.
 
\end{proof}

By reasoning analogous to the proof of Theorem~\ref{Prop - 1 - app} and applying Theorem~\ref{Prop2}(2), we obtain the following theorems.
\begin{theorem}\label{Prop - 2 - app} 
Let  $p,r\in \mathcal{P}(\mathbb{R}^n)$ and $1\leq \nu \leq p^-$, then 
\begin{enumerate}
    \item  $\left\|
    g_{\alpha,t }*
    \varphi\right\|_{L^{p(\cdot)}(\mathbb{R}^n)}
\leq
C  t^{\frac{-n  \omega(t)}{2\alpha} }\|\varphi\|_{L^{r(\cdot)}(\mathbb{R}^n)\cap L^{\nu}(\mathbb{R}^n)}$,
\item $\left\|(-\Delta)^{\kappa / 2} g_{\alpha,t } * 
\varphi\right\|_{L^{p(\cdot)}(\mathbb{R}^n)} \leq C t^{-\frac{\kappa}{2 \alpha}-\frac{n  \omega(t)}{2\alpha}}\|\varphi\|_{L^{r(\cdot)}(\mathbb{R}^n)\cap L^{\nu}(\mathbb{R}^n)}$,
\end{enumerate}
where
\begin{equation*}
\omega(t)=
\begin{cases} \frac{1}{r^-}(1-\frac{\nu}{p^+}) , & \text{if }\ 0<t \leq 1 \ , \\
\frac{1}{r^+}(1-\frac{\nu}{p^-}) , & \text{if }\  t > 1 \ .
\end{cases}
\end{equation*}

\end{theorem}
\begin{theorem}\label{Prop - 3 - app} 
Let  $p\in \mathcal{P}(\mathbb{R}^n)$ and $r\in \mathcal{P}^{\log}(\mathbb{R}^n)$  such that  $r_\infty=r^- \leq p^-$, then 
\begin{enumerate}
    \item  $\left\|
    g_{\alpha,t }*
    \varphi\right\|_{L^{p(\cdot)}(\mathbb{R}^n)}
\leq
C  t^{\frac{-n  \psi(t)}{2\alpha} }\|\varphi\|_{L^{r(\cdot)}(\mathbb{R}^n)}$,
\item $\left\|(-\Delta)^{\kappa / 2} g_{\alpha,t } * 
\varphi\right\|_{L^{p(\cdot)}(\mathbb{R}^n)} \leq C t^{-\frac{\kappa}{2 \alpha}-\frac{n  \psi(t)}{2\alpha}}\|\varphi\|_{L^{r(\cdot)}(\mathbb{R}^n)}$,
\end{enumerate}
where
\begin{equation*}
\psi(t)=
\begin{cases} \frac{1}{r^-}(1-\frac{r^-}{p^+}) , & \text{if } 0<t \leq 1 \ ; \\
\frac{1}{r^+}(1-\frac{r^-}{p^-}) , & \text{if }  t > 1 .
\end{cases}
\end{equation*}
\end{theorem}
\begin{remark}
When $p^+ = p_\infty$ or $r^- = r_\infty$ with globally $\log$-H\"{o}lder continuous exponents, Theorems \ref{Prop - 1 - app} and \ref{Prop - 3 - app} provide a direct generalization of Lemma \ref{Lemma - MIAO}. In contrast, Theorem \ref{Prop - 2 - app} establishes a generalization in the intersection space $L^{r(\cdot)}(\mathbb{R}^n) \cap L^{\nu}(\mathbb{R}^n)$ without requiring such conditions. Notably, if the exponents $p, r,$ and $\nu$ are constants with $\nu = p$, then $\sigma(t) = \omega(t)=\psi(t) =  (\frac{1}{r} - \frac{1}{p})$. Consequently, Theorems \ref{Prop - 1 - app}, \ref{Prop - 2 - app} and \ref{Prop - 3 - app}  recover the classical results of Theorem~\ref{Lemma - MIAO} with the same optimal constants $  (\frac{1}{r} - \frac{1}{p})$.
\end{remark}

We conclude this section with the following estimates, which follow directly from the preceding proofs.
\begin{theorem} \label{THm - eta Lr-Lp conv} Let $m>n$, we have the following inequalities 
\begin{enumerate}
\item If $p\in \mathcal{P}^{\log}(\mathbb{R}^n)$ and $r\in \mathcal{P}(\mathbb{R}^n)$ such that  $r^+\leq p^+=p_\infty$, then 
\begin{align*}
\|\eta_{t,m}*f\|_{L^{p(\cdot)}(\mathbb{R}^n)}
\leq C t^{-n\sigma(t)} \|f\|_{L^{ r(\cdot)}(\mathbb{R}^n)}\ , \forall t\in (0,+\infty).
\end{align*}
\item if  $p\in \mathcal{P}(\mathbb{R}^n)$ and $r\in \mathcal{P}^{\log}(\mathbb{R}^n)$  such that  $r_\infty=r^- \leq p^-$ then 
\begin{align*}
\left\|
   \eta_{t,m}*f \right\|_{L^{p(\cdot)}(\mathbb{R}^n)}
\leq
C  t^{-n  \psi (t) }\|f \|_{L^{r(\cdot)}(\mathbb{R}^n)} \ , \forall t\in (0,+\infty).
\end{align*}
\end{enumerate} 
Where $\sigma(t), \omega (t), \psi (t)$ are given above in Theorems \ref{Prop - 1 - app} and \ref{Prop - 3 - app}, respectively.
\end{theorem}

\bigskip
\subsection{Existence of Local Mild Solutions}\label{Sec-Exis-local}
In this section, we establish several versions of local mild solution theorems within various mixed-norm Lebesgue spaces.
Note that while the spatial convolution in $L^{q(x)}(\mathbb{R}^3)$ from Section \ref{sec-conv} requires $q(x)$ to satisfy $\log$-H\"older continuity, the current space--time framework $L^{p(t)}$ used in this section is free from such constraints.  In the following section, we will further sharpen these results by employing the Riesz potential operator,  which will again require $\log$-H\"older continuity of the exponent $p(t)$.

In the following, we consider two variable exponents \( p(t) \in \mathcal{P}(\,(0, T)\,) \) where  \( T \in (0, +\infty) \) and $q(x)\in \mathcal{P}(\mathbb{R}^n)$, and  the mixed-norm   function space
\begin{equation*}
L^{p(\cdot)}\!\big(0, T;\, L^{q(\cdot)}(\mathbb{R}^3)\big),
\end{equation*}
 this space is also denoted \(L^{p(t)}(0,T\, , L^{q(x)}(\mathbb{R}^n)) \) is endowed with the Luxemburg norm defined by
\begin{equation*}
\| u \|_{L^{p(\cdot)}\!(0, T;\, L^{q(\cdot)}(\mathbb{R}^3))}= \left\|\ \|\, u(x,t)\, \|_{L^{q(x)}(\mathbb{R}^2)}\ \right\|_{L^{p(t)}((0,T))}
\end{equation*}
Equivalently, this norm  is  expressed as
\begin{equation*}
\| u \|_{L^{p(\cdot)}\!(0, T;\, L^{q(\cdot)}(\mathbb{R}^3))}
= \inf \left\{
\lambda > 0 :
\int_{0}^{T}
\left(
\frac{ \| u(t, \cdot) \|_{L^{q(\cdot)}(\mathbb{R}^3)} }{ \lambda }
\right)^{p(t)}
dt
\le 1
\right\}.
\end{equation*}

\begin{theorem}\label{THM1 - Local - Existence} 
Let $\alpha\in (\frac{1}{2},1]$,  \(p \in \mathcal{P}([0,1])\) and  \(q \in \mathcal{P}^{\log}(\mathbb{R}^3)\) with $2<q^-\leq q^+ <+\infty$, $q^+=q_\infty$, $p^->2$  and 
\begin{equation}
 \frac{2\alpha}{p^-}+\frac{3\vartheta_1}{2} < \alpha-\frac{1}{2}\, \text{ where } \, 
\vartheta_1 : = \left(\frac{2}{q^-}-\frac{1}{q_\infty} \right)
 \left(\frac{2+q_\infty(1-\frac{2}{q^-})}{1+q_\infty(1-\frac{2}{q^-})}- \frac{1}{q_\infty-1}\right)\label{Part_def of cont rho}
\end{equation}
is the value corresponding to $q$ from (\ref{rho-label}) on the interval $(0,1]$. If 
\begin{enumerate}
\item $ f\in L^{1}\big(0,1;L^{q(\cdot)}(\mathbb{R}^{3})\big)$ is a divergence-free exterior force,  and
\item  $ u_{0}\in L^{q(\cdot)}(\mathbb{R}^{3})$ is a divergence-free initial data,
\end{enumerate}
 then there exists a time $T\in(0,1]$ and a unique mild solution of the fractional Navier--Stokes equations \eqref{NS_Intro} in the space $L^{p(\cdot)}\big(0,T;L^{q(\cdot)}(\mathbb{R}^{3})\big)$.
\end{theorem}
\begin{proof}We begin by estimating the initial data term,
\begin{align*}
e_0(t,x) = g_{\alpha, t} * u_0(x)+ \int_{0}^{t} g_{\alpha, t-s} * f(s, x)\, ds
\end{align*}
in $L^{p(\cdot)}\big(0,T;L^{q(\cdot)}(\mathbb{R}^{3})\big)$ with $T\in(0, 1]$. Since 
 $q\in \mathcal{P}^{\log}(\mathbb{R}^{3})$, by virtue of the estimates in (\ref{est-g_eta}) and   (\ref{eneq p norm}) we have  
 \begin{equation*}
 \|g_{\alpha, t} * u_0\|_{L^{q(\cdot)}(\mathbb{R}^3)}
\leq \left\| \eta _{t^{\frac{1}{2\alpha} },3 + 2\alpha}*|u_0| \ \right\|_{L^{q(\cdot)}(\mathbb{R}^3)} 
 \leq C \| u_0\|_{L^{q(\cdot)}(\mathbb{R}^3)},
 \end{equation*}
then applying the $L^{p(\cdot)}$-norm yields
\begin{align*}
\|g_{\alpha, t} * u_0\|_{L^{p(t)}(0,T;L^{q(\cdot)}(\mathbb{R}^{3}))}\leq C \|1 \|_{L^{p(\cdot)}((0,1))} \|u_0\|_{L^{q(\cdot)}(\mathbb{R}^3)}
\leq C  \|u_0\|_{L^{q(\cdot)}(\mathbb{R}^3)}.
\end{align*}
Applying Minkowski's inequality and the estimate in \eqref{eneq p norm}, the second term of $e_0$ can be estimated as follows, 
 \begin{align*}
 \left\| \int_{0}^{t} g_{\alpha, t-s} * f(s, x)\, ds \right\|_{L^{q(x)}(\mathbb{R}^3)} &\leq C 
 \int_{0}^{t} \left\|g_{\alpha, t-s} * f(s, x)\right\|_{L^{q(x)}(\mathbb{R}^3)}\, ds \\
 &\leq C \int_{0}^{t} \left\|\eta _{(t-s)^{ \frac{1}{2\alpha}},3 + 2\alpha} *| f(s, x)|\right\|_{L^{q(x)}(\mathbb{R}^3)}\, ds \\
 &\leq  C \int_{0}^{t} \|f(s, x)\|_{L^{q(x)}(\mathbb{R}^3)}\, ds,
 \end{align*}
 which implies that
 \begin{align*}
 \left\| \left\| \int_{0}^{t} g_{\alpha, t-s} * f(s, x)\, ds \right\|_{L^{q(x)}(\mathbb{R}^3)} \right\|_{L^{p(t)}((0,T))} &\leq  C \|1 \|_{L^{p(t)}((0,1))} \int_{0}^{T } \|f(s, x)\|_{L^{q(x)}(\mathbb{R}^3)}\, ds,
 \end{align*}
 with mentioning that $\|1 \|_{L^{p(t)}((0,1))}\leq C$ for some positive $C$ independent of $T$  by inequality~(\ref{1 in bound in P}). Consequently, $e_0\in L^{p(\cdot)}\big(0,T;L^{q(\cdot)}(\mathbb{R}^{3})\big)$ and we arrive  at the estimate,
\begin{align}\label{est1-e0}
\|e_0\|_{L^{p(\cdot)}(0,T;L^{q(\cdot)}(\mathbb{R}^{3}))} \leq C_1 \left(\left\|u_0 \right\|_{L^{q(\cdot)}(\mathbb{R}^3)}+ 
\left\| \| f(\cdot, x) \|_{L^{q(x)}(\mathbb{R}^3)} \right\|_{L^{1}((0,1))}
\right),
\end{align}
where $C_1$ is independent of $T$. Next, we estimate the bilinear operator 
\begin{align*}
B(u,u)(x,t) =  -\int_{0}^t K_{\alpha,t-s}\ast (u \otimes u)(s, x)\, ds.
\end{align*}
Utilizing the relation \eqref{eqB=F} along with the estimate \eqref{eneq-F} and Minkowski's inequality, it follows that
\begin{align*}
\left\|  \int_{0}^t
K_{\alpha,t-s}\ast (u \otimes u)(s, x)
\, ds\right\|_{L^{q(x)}(\mathbb{R}^3)} &\leq C 
\max_{j=1,2,3} \int_{0}^{t} \sum_{h,k=1}^{3} \left\|  
K^{j;h,k}_{\alpha,t-s} \ast \left(u_h \, u_k\right) \, \right\|_{L^{q(\cdot)}(\mathbb{R}^3)}(s) ds \\
&\leq C 
 \int_{0}^{t}\frac{1}{(t-s)^{\frac{1}{ 2\alpha}}} \sum_{h,k=1}^{3} \left\| 
 \eta_{(t-s)^{\frac{1}{2\alpha}},4} \ast \left|u_h \, u_k\right| \, \right\|_{L^{q(\cdot)}(\mathbb{R}^3)}(s)\, ds .
\end{align*}
Now, in view of the assumptions on $q(\cdot)$, applying Lemma~\ref{LemmaCom 1} with the fact that  $0\leq t-s\leq 1$ for all $s,t\in(0,1)$ and H\"{o}lder's inequality, it follows for all $t\in(0,1]$, $0<s <t$ , and $h, k \in \{1, 2, 3\}$ that,
\begin{align*}
 \left\| \eta_{(t-s)^{\frac{1}{2\alpha}},n+1} \ast \left|u_h \, u_k\right| \, \right\|_{L^{q(\cdot)}(\mathbb{R}^3)} &\leq C (t-s)^{\frac{-3\vartheta_1}{2\alpha}} 
  \left\|  u_h \, u_k \, \right\|_{L^{\frac{q(\cdot)}{2}}(\mathbb{R}^3)} \\ 
  &\leq C (t-s)^{\frac{-3\vartheta_1}{2\alpha}}  \left\|  u_k \, \right\|_{L^{q(\cdot)}(\mathbb{R}^3)} \left\|  u_h \, \right\|_{L^{q(\cdot)}(\mathbb{R}^3)}
\end{align*}
where $\vartheta_1$ is given in (\ref{Part_def of cont rho}), we conclude that 
\begin{align}\label{enequa THM1}
\left\|   \int_{0}^t K_{\alpha,t-s}\ast (u \otimes u)(s, x)\, ds \right\|_{L^{q(x)}(\mathbb{R}^3)} &\leq C \int_{0}^{t}\frac{1}{(t-s)^{\frac{1+3\vartheta_1}{2\alpha}}} \| u(s, x) \|_{L^{q(x)}(\mathbb{R}^3)}^2 \,ds.
\end{align}
Now, since $p^- > 2$, we define the auxiliary exponent $\widetilde{p}(t)$ by the relation
$\frac{1}{\widetilde{p}(t)}+\frac{2}{p(t)}=1,t\in(0,1)$. Therefore,    $\widetilde{p}(\cdot)$ is bounded and 
 $\widetilde{p}^+= \frac{p^-}{p^- -2}.$
Now, for a fixed $t\in (0,1)$, by H\"{o}lder's inequality,  
\begin{align*}
\int_{0}^{t}\frac{\left\| u(s, x) \right\|_{L^{q(x)}(\mathbb{R}^3)}^2}{(t-s)^{\frac{1+3\vartheta_1}{2\alpha}}}  \,ds \leq \left\|\| u(s, x) \|_{L^{q(x)}(\mathbb{R}^3)}\right\|_{L^{p(s)}((0,t))}^2 
\left\|(t-s)^{-\frac{1+3\vartheta_1}{2\alpha}}\right\|_{L^{\widetilde{p}(s)}((0,t))}.
\end{align*} 
According to the hypotheses of the theorem, we have $\frac{2\alpha}{p^-}+\frac{3\vartheta_1}{2} < \alpha-\frac{1}{2}$, which implies that $0<\widetilde{p}^+ \frac{1+3\vartheta_1}{2\alpha}<1$, thus,   we have the following  estimate, for every $0<t\leq 1$,
\begin{align*}
\varrho_{\widetilde{p}(s),(0,t)}\left( (t-s)^{-\frac{1+3\vartheta_1}{2\alpha}} \right)= \int_0^t (t-s)^{-\widetilde{p}(s)\frac{1+3\vartheta_1}{2\alpha}}\,ds \leq \int_0^t (t-s)^{-\widetilde{p}^+\frac{1+3\vartheta_1}{2\alpha}}\,ds = 
\frac{t^{1-\widetilde{p}^+\frac{1+3\vartheta_1}{2\alpha}}}{1-\widetilde{p}^+\frac{1+3\vartheta_1}{2\alpha}}.
\end{align*}
Therefore, we arrive at the following estimate
\begin{align*}
\left\|\, (t-s)^{-\frac{1+3\vartheta_1}{2\alpha}}\,\right\|_{L^{\widetilde{p}(s)}((0,t))} \leq C T^{\,\delta},\, \delta:=\frac{1}{\widetilde{p}^+}-\frac{1+3\vartheta_1}{2\alpha},
\end{align*}
for every $0<t\leq T$. It follows that there exists $C>0$ independent of $T$ such that,  for every $t\in(0,T]$,
\begin{align*}
\int_{0}^{t}\frac{\| u(s, x) \|_{L^{q(x)}(\mathbb{R}^3)}^2}{(t-s)^{\frac{1+3\vartheta_1}{2\alpha}}}  \,ds \leq C T^{\,\delta} \left\|\| u(t, x) \|_{L^{q(x)}(\mathbb{R}^3)}\right\|_{L^{p(t)}((0,T))}^2.
\end{align*}
Finally, applying the $L^{p(\cdot)}$-norm we conclude that 
\begin{align*}
\left\| \left\|  \int_{0}^t K_{\alpha,t-s}\ast (u \otimes u)(s, x)\, ds\right\|_{L^{q(x)}(\mathbb{R}^3)} 
\right\|_{L^{p(t)}((0,T))} &\leq C_2 T^{\,\delta} \left\|\| u(t, x) \|_{L^{q(x)}(\mathbb{R}^3)}\right\|_{L^{p(t)}((0,T))}^2.
\end{align*}
with a constant $C_2 > 0$ independent of $T$. Therefore, we obtain the following estimate for the bilinear operator $B(u,u)$ in the space $L^{p(\cdot)}\big(0,T;L^{q(\cdot)}(\mathbb{R}^{3})\big)$,
\begin{align}\label{est1-E(u,u)}
\|  B(u,u)\|_{L^{p(\cdot)}(0,T;L^{q(\cdot)}(\mathbb{R}^{3}))}\ \leq C_2 T^{\,\delta} \ 
 \| u \|_{L^{p(\cdot)}(0,T;L^{q(\cdot)}(\mathbb{R}^{3}))} \| u \|_{L^{p(\cdot)}(0,T;L^{q(\cdot)}(\mathbb{R}^{3}))}.
\end{align}
 To apply  Banach-Picard principle - Theorem~\ref{Thm:BanachPicard} - we may choose $T\in (0,1]$ small enough such that 
\begin{align*}
\|e_0\|_{L^{p(\cdot)}(0,T;L^{q(\cdot)}(\mathbb{R}^{3}))} \leq C_1 \left(\left\|u_0 \right\|_{L^{q(\cdot)}(\mathbb{R}^3)}+ 
\left\| \| f(\cdot, x) \|_{L^{q(x)}(\mathbb{R}^3)} \right\|_{L^{1}((0,1))}
\right) <\frac{1}{4C_2 T^{\,\delta}}.
\end{align*}
The proof of the theorem is thus complete.

\end{proof}

In Theorem \ref{THM1 - Local - Existence}  the exponent $q(x)$ is bounded and $p(t)$ may be unbounded. In   the following theorem we deal with the case where $q_\infty=\infty$, the proof proceeds similarly to that of Theorem~\ref{THM1 - Local - Existence}, with the necessary modifications provided by Lemma~\ref{LemmaCom 2}.
\begin{theorem}\label{THM1.1 - Local - Existence}
Let $\alpha\in (\frac{1}{2},1]$,  \(p \in \mathcal{P}([0,1])\) and  \(q \in \mathcal{P}^{\log}(\mathbb{R}^3)\) with $ q^- > 2$, $q_\infty=\infty$, $p^->2$ and $\frac{2\alpha}{p^-}+\frac{3}{q^-}<\alpha-\frac{1}{2}$. If 
\begin{enumerate}
\item $ f\in L^{1}\big(0,1;L^{q(\cdot)}(\mathbb{R}^{3})\big)$ is a divergence-free exterior force,  and
\item  $ u_{0}\in L^{q(\cdot)}(\mathbb{R}^{3})$ is a divergence-free initial data,
\end{enumerate}
 then there exists a time
$T\in(0,1]$ and a unique mild solution of the fractional Navier-Stokes equations \eqref{NS_Intro}
in the space $L^{p(\cdot)}\big(0,T;L^{q(\cdot)}(\mathbb{R}^{3})\big)$.
\end{theorem}
\begin{remark}
Under the assumptions of  Theorems~\ref{THM1 - Local - Existence} and \ref{THM1.1 - Local - Existence} , since we have $\alpha\in (\frac{1}{2},1]$ and $\vartheta_1\geq \frac{1}{q^-}$ by Remark~\ref{Remark on mq},  thus,   the condition $ \frac{2\alpha}{p^-}+\frac{3\vartheta}{2} < \alpha-\frac{1}{2}$ implies that $q^->\frac{3}{2\alpha-1}\geq 3$ and $p^->\frac{4\alpha}{2\alpha-1}\geq 4$.
\end{remark}

\begin{theorem}\label{THM2 - Local - Existence}
Let $\alpha\in (\frac{1}{2},1]$, \(q \in \mathcal{P}^{\log}(\mathbb{R}^3)\) and \(p(\cdot) \in \mathcal{P}([0,1])\) with 
$p^-> \frac{4\alpha}{2\alpha-1}$. 
If 
\begin{enumerate}
\item $ f\in L^{1}\big(0,1;L^{q(\cdot)}(\mathbb{R}^{3})\cap L^\infty(\mathbb{R}^3)\big)$ is a divergence-free exterior force, 
and
\item $ u_{0}\in L^{q(\cdot)}(\mathbb{R}^{3})$ is a divergence-free initial data,
\end{enumerate}
 then there exists a time $T\in(0,1]$ and a unique mild solution of the fractional Navier-Stokes equations \eqref{NS_Intro} in the space $L^{p(\cdot)}\big(0,T;L^{q(\cdot)}(\mathbb{R}^{3})\cap L^\infty(\mathbb{R}^3)\big)$.
\end{theorem}
\begin{proof} Denote $\mathcal{E}_T^{p(\cdot),(q(\cdot),\infty)}=L^{p(\cdot)}\big(0,T;L^{q(\cdot)}(\mathbb{R}^{3})\cap L^\infty(\mathbb{R}^3)\big)$. 
We begin by estimating the initial data term $e_0$ in $L^{p(\cdot)}\big(0,T;L^{q(\cdot)}(\mathbb{R}^{3})\cap L^\infty(\mathbb{R}^3)\big)$ with $\ T\leq 1$. Since 
 $q\in \mathcal{P}^{\log}(\mathbb{R}^{3})$.  by the estimates in (\ref{est-g_eta}) and   (\ref{eneq p norm}) we have  
 \begin{equation*}
 \|g_{\alpha, t} * u_0\|_{L^{q(\cdot)}(\mathbb{R}^3)\cap L^\infty(\mathbb{R}^3)}
\leq\left\| \eta _{t^{\frac{1}{2\alpha} },3 + 2\alpha}*|u_0| \ \right\|_{L^{q(\cdot)}(\mathbb{R}^3)\cap L^\infty(\mathbb{R}^3)} 
 \leq C \| u_0\|_{L^{q(\cdot)}(\mathbb{R}^3)\cap L^\infty(\mathbb{R}^3)},
\end{equation*}   
for all $0<t\leq T \leq 1$, then applying the  $L^{p(\cdot)}$-norm yields
\begin{align*}
\|g_{\alpha, t} * u_0\|_{\mathcal{E}_T^{p(t),(q(\cdot),\infty)}}\leq C \|u_0\|_{L^{q(\cdot)}(\mathbb{R}^3)\cap L^\infty(\mathbb{R}^3)}.
\end{align*}
Applying Minkowski's inequality and the estimates in (\ref{est-g_eta}) and   (\ref{eneq p norm}),  we arrive at the following bound
 \begin{align*}
 \left\| \int_{0}^{t} g_{\alpha, t-s} * f(s, x)\, ds \right\|_{L^{q(x)}(\mathbb{R}^3)\cap L^\infty(\mathbb{R}^3)} &\leq C 
 \int_{0}^{t} \left\|g_{\alpha, t-s} * f(s, x)\right\|_{L^{q(x)}(\mathbb{R}^3)\cap L^\infty(\mathbb{R}^3)}\, ds \\
 &\leq C \int_{0}^{t} \left\|\eta _{(t-s)^{ \frac{1}{2\alpha}},3 + 2\alpha} *| f(s, x)|\right\|_{L^{q(x)}(\mathbb{R}^3)\cap L^\infty(\mathbb{R}^3)}\, ds \\
 &\leq  C \int_{0}^{t} \|f(s, x)\|_{L^{q(x)}(\mathbb{R}^3)\cap L^\infty(\mathbb{R}^3)}\, ds,
 \end{align*}
from which we deduce that
 \begin{align*}
 \left\| \left\| \int_{0}^{t} g_{\alpha, t-s} * f(s, x)\, ds \right\|_{L^{q(x)}(\mathbb{R}^3)\cap L^\infty(\mathbb{R}^3)} \right\|_{L^{p(t)}((0,T))} &\leq  C \left\| \| f(\cdot, x) \|_{L^{q(x)}(\mathbb{R}^3)\cap L^\infty(\mathbb{R}^3)} \right\|_{L^{1}((0,1))}.
 \end{align*}
These bounds lead to the following estimate for the  term $e_0$
\begin{align}\label{est1-e0}
\|e_0\|_{\mathcal{E}_T^{p(\cdot),(q(\cdot),\infty)}} \leq C_1 \left(\left\|u_0 \right\|_{L^{q(\cdot)}(\mathbb{R}^3)\cap L^\infty(\mathbb{R}^3)}+ 
\left\| \| f(\cdot, x) \|_{L^{q(x)}(\mathbb{R}^3)\cap L^\infty(\mathbb{R}^3)} \right\|_{L^{1}((0,1))}
\right)
\end{align}
where $C_1$ is independent of $T$. We now turn to the estimate for the  operator $B(u, u)$,

\begin{align*}
&\left\|  \int_{0}^tK_{\alpha,t-s}\ast (u \otimes u)(s, x)\, ds\right\|_{L^{q(x)}(\mathbb{R}^3)\cap L^\infty(\mathbb{R}^3)} \\
&\hspace{3cm}\leq C 
 \int_{0}^{t}\frac{1}{(t-s)^{\frac{1}{2\alpha}}} \sum_{h,k=1}^{3} \left\| 
 \eta_{(t-s)^{\frac{1}{2\alpha}},4} \ast \left|u_h \, u_k\right| \, \right\|_{L^{q(\cdot)}(\mathbb{R}^3)\cap L^\infty(\mathbb{R}^3)}(s) \,ds 
\end{align*}
 for all $0<s\leq t\leq 1$. Again, since $q\in \mathcal{P}^{\log}(\mathbb{R}^{3})$, by  the estimate in  (\ref{eneq p norm}) we have 
\begin{align*}
\left\| 
 \eta_{(t-s)^{\frac{1}{2\alpha}},4} \ast \left|u_h \, u_k\right| \, \right\|_{L^{q(\cdot)}(\mathbb{R}^3)} &\leq 
 \| u_h \|_{L^\infty(\mathbb{R}^3)} \left\| 
 \eta_{(t-s)^{\frac{1}{2\alpha}},n+1} \ast \left|  u_k\right| \, \right\|_{L^{q(\cdot)}(\mathbb{R}^3)} \\
 &\leq c  \| u_h \|_{L^\infty(\mathbb{R}^3)} \left\| 
  u_k \, \right\|_{L^{q(\cdot)}(\mathbb{R}^3)} \\
  &\leq c \| u \|_{L^{q(\cdot)}(\mathbb{R}^3)\cap L^\infty(\mathbb{R}^3)} \left\| 
  u \, \right\|_{L^{q(\cdot)}(\mathbb{R}^3)\cap L^\infty(\mathbb{R}^3)}
\end{align*}
Similarly 
\begin{align*}
\left\| 
 \eta_{(t-s)^{\frac{1}{2\alpha}},4} \ast \left|u_h \, u_k\right| \, \right\|_{L^{\infty}(\mathbb{R}^3)} 
&\leq \| u_h\, u_k \|_{ L^\infty(\mathbb{R}^3)} \left\|  \eta_{(t-s)^{\frac{1}{2\alpha}},4} \right\|_{ L^{1}(\mathbb{R}^3)}\\
& \leq C \| u \|_{L^{q(\cdot)}(\mathbb{R}^3)\cap L^\infty(\mathbb{R}^3)} \left\| 
  u \, \right\|_{L^{q(\cdot)}(\mathbb{R}^3)\cap L^\infty(\mathbb{R}^3)} \ .
\end{align*}
Consequently, we obtain
\begin{align}\label{FromLocToGlob}
\left\|  \int_{0}^tK_{\alpha,t-s}\ast (u \otimes u)(s, x)\, ds\right\|_{L^{q(x)}(\mathbb{R}^3)\cap L^\infty(\mathbb{R}^3)}\leq C 
 \int_{0}^{t}\frac{\,\| u \|_{L^{q(\cdot)}(\mathbb{R}^3)\cap L^\infty(\mathbb{R}^3)}^2}{(t-s)^{\frac{1}{2\alpha}}} \, ds.
\end{align}

Since $\alpha\in (\frac{1}{2},1]$ and $p^-> \frac{4\alpha}{2\alpha-1}$  then   $p^-\geq 4>2$ and $ \frac{\widetilde{p}^+}{2\alpha}<1$ where  $\widetilde{p}$ is defined by the relation
$\frac{1}{\widetilde{p}}+\frac{2}{p}=1$, thus   for a fixed $t\in (0,1]$  we have  
\begin{align*}
\int_{0}^{t}\frac{\,\| u \|_{L^{q(\cdot)}(\mathbb{R}^3)\cap L^\infty(\mathbb{R}^3)}^2 }{(t-s)^{\frac{1}{2\alpha}}} \,ds \leq \left\|\| u(s, x) \|_{L^{q(x)}(\mathbb{R}^3)\cap L^\infty(\mathbb{R}^3)}\right\|_{L^{p(x)}((0,t))}^2 
\left\|(t-s)^{-\frac{1}{2\alpha}}\right\|_{L^{\widetilde{p}(s)}((0,t))}
\end{align*}
and
\begin{align*}
\varrho_{\widetilde{p}(s),(0,t)}\left( (t-s)^{-\frac{1}{2\alpha}} \right)= \int_0^t (t-s)^{-\widetilde{p}(s)\frac{1}{2\alpha}}\,ds \leq \int_0^t (t-s)^{-\frac{\widetilde{p}^+}{2\alpha}}\,ds = 
\frac{t^{1-\frac{\widetilde{p}^+}{2\alpha}}}{1-\frac{\widetilde{p}^+}{2\alpha}},
\end{align*}
where $\widetilde{p}^+= \frac{p^-}{p^- -2}$ . It follows that  for all $0<t\leq T\leq 1$,
\begin{align*}
\int_{0}^{t}\frac{1}{(t-s)^{\frac{1}{2\alpha}}} \| u(s, x) \|_{L^{q(x)}(\mathbb{R}^3)\cap L^\infty(\mathbb{R}^3)}^2 \,ds \leq C T^{\,\delta} \left\|\| u(\cdot, x) \|_{L^{q(x)}(\mathbb{R}^3)\cap L^\infty(\mathbb{R}^3)}\right\|_{L^{p(\cdot)}((0,T))}^2 ,
\end{align*}
where $\delta:=\frac{1}{\widetilde{p}^+}-\frac{1}{2\alpha}$. Finally, applying the $L^{p(\cdot)}$-norm we conclude that 
\begin{align}\label{est1-E(u,u)}
\|  B(u,u)\|_{\mathcal{E}_T^{p(\cdot),(q(\cdot),\infty)}}\ \leq C_2 T^{\,\delta} \  \| u \|_{\mathcal{E}_T^{p(\cdot),(q(\cdot),\infty)}} \| u \|_{\mathcal{E}_T^{p(\cdot),(q(\cdot),\infty)}},
\end{align}
where $C_2$ is independent of $T$. The proof is completed by invoking the Banach-Picard fixed-point theorem by selecting a sufficiently small  $T\in(0,1]$, thereby ensuring the existence of a unique mild solution in the space $L^{p(\cdot)}\big(0,T;L^{q(\cdot)}(\mathbb{R}^{3})\cap L^\infty(\mathbb{R}^3)\big)$.

\end{proof}
\begin{theorem} \label{THM3 - Local - Existence}
Let $\alpha\in (\frac{1}{2},1]$, \(q \in \mathcal{P}^{\log}(\mathbb{R}^3)\) with $q^-\geq 2$,  $3<\nu \leq  2q^-$ and \(p(\cdot) \in \mathcal{P}((0,1))\) with $p^->2$. If 
$\frac{2\alpha}{p^-}+\frac{1}{2}\max\left(\frac{6}{q^-}(1-\frac{\nu}{2q^+}),\frac{3}{\nu} \right) < \alpha-\frac{1}{2}$ and 
\begin{enumerate}
\item $ f\in L^{1}\big(0,1;L^{q(\cdot)}(\mathbb{R}^3)\cap L^\nu(\mathbb{R}^3)\big)$ is a divergence-free exterior force, 
and
\item $ u_{0}\in L^{q(\cdot)}(\mathbb{R}^3)\cap L^\nu(\mathbb{R}^3)$ is a divergence-free initial data,
\end{enumerate}
 then there exists a time $T\in(0,1]$ and a unique mild solution of the fractional Navier-Stokes equations \eqref{NS_Intro}
in the space $L^{p(\cdot)}\big(0,T;L^{q(\cdot)}(\mathbb{R}^3)\cap L^\nu(\mathbb{R}^3)\big)$.
\end{theorem}
\begin{proof}The proof is similar to that of Theorem \ref{THM2 - Local - Existence}.  By the estimates in (\ref{est-g_eta}) and   (\ref{eneq p norm}) we have  
 \begin{equation*}
 \|g_{\alpha, t} * u_0\|_{L^{q(\cdot)}(\mathbb{R}^3)\cap L^\nu(\mathbb{R}^3)}
\leq\left\| \eta _{t^{\frac{1}{2\alpha} },3 + 2\alpha}*|u_0| \ \right\|_{L^{q(\cdot)}(\mathbb{R}^3)\cap L^\nu(\mathbb{R}^3)} 
 \leq C \| u_0\|_{L^{q(\cdot)}(\mathbb{R}^3)\cap L^\nu(\mathbb{R}^3)},
\end{equation*}   
similarly, we estimate the second term of $e_0$. By the assumptions on $q(\cdot)$ we can apply  Lemma \ref{Lemma of LV cap Lp}, which yields, for all $0<s<t<1$,
\begin{align*}
\left\| 
 \eta_{(t-s)^{\frac{1}{2\alpha}},3} \ast \left|u_h \, u_k\right| \, \right\|_{L^{q(\cdot)}(\mathbb{R}^3)\cap L^\nu(\mathbb{R}^3)} &\leq 
C(t-s)^{-\frac{\max\left(3\omega,\frac{3}{\nu} \right)}{2\alpha} } 
\| u_h \|_{L^{q(\cdot)}(\mathbb{R}^3)\cap L^\nu(\mathbb{R}^3)} \left\| 
  u_k \, \right\|_{L^{q(\cdot)}(\mathbb{R}^3)\cap L^\nu(\mathbb{R}^3)}
\end{align*}
where $\omega:=\frac{2}{q^-}(1-\frac{\nu}{2q^+})$, then it follows
\begin{align*}
\left\|  \int_{0}^tK_{\alpha,t-s}\ast (u \otimes u)(s, x)\, ds\right\|_{L^{q(x)}(\mathbb{R}^3)\cap L^\nu(\mathbb{R}^3)} \leq C \int_{0}^{t}\frac{\,\| u(s, x) \|_{L^{q(x)}(\mathbb{R}^3)\cap L^\nu(\mathbb{R}^3)}^2}{(t-s)^{\gamma}}  \,ds,
\end{align*}
for all $ 0<t\leq 1 $, where $\gamma:=\frac{1}{2\alpha}\left(1+\max\left(\frac{6}{q^-}(1-\frac{\nu}{2q^+}),\frac{3}{\nu} \right)\right)$,  and 
with similar arguments as in the proof of Theorem \ref{THM2 - Local - Existence} we conclude the proof.

\end{proof}
\begin{remark}We give some remarks based on the condition $\frac{2\alpha}{p^-}+\frac{1}{2}\max\left(\frac{6}{q^-}(1-\frac{\nu}{2q^+}),\frac{3}{\nu} \right) < \alpha-\frac{1}{2}$.
The condition involves   $\max(\frac{6}{q^-}(1-\frac{\nu}{2q^+}),\frac{3}{\nu} )$ 
which may be true for large enough $q^-$ and $\nu$, we see that $ \nu>\frac{3}{2\alpha-1}\geq 3$ which yields that $2q^->\frac{3}{2\alpha-1}\geq 3$. It possible to take  $q^-=2$, for example if  $q^-=2$ and $\nu=\frac{7}{2}$ then $\frac{6}{q^-}(1-\frac{\nu}{2q^+})=3(1-\frac{7}{4q^+})$ and
 $\frac{3}{\nu}= \frac{6}{7}$ then if $q^+ \leq \frac{49}{20}$ we get 
  $\max(\frac{6}{q^-}(1-\frac{\nu}{2q^+}),\frac{3}{\nu} )=\frac{6}{7}$.
Then the condition becomes    $\frac{2\alpha}{p^-}+\frac{3}{7}< \alpha-\frac{1}{2}$ or equivalently, $p^- > \frac{28\alpha}{14\alpha - 13}$ provided that $1\geq \alpha > \frac{13}{14}$.
\end{remark}

By applying the estimate established in (\ref{Omega-Lp-Log}) and invoking arguments similar to those used above, we establish Theorem~\ref{THM4 - Local - Existence} for the case $q^- = q_\infty$, thereby providing a complementary result to the regime $q^+ = q_\infty$ established in Theorem~\ref{THM1 - Local - Existence}. In summary, we have addressed three primary cases: the scenarios $q_\infty\in\{q^-,q^+\}$, and the generalized setting $L^{q(\cdot)}(\mathbb{R}^3)\cap L^\nu(\mathbb{R}^3\big)$ presented in Theorem~\ref{THM3 - Local - Existence}. By working in the intersection $L^{q(\cdot)}(\mathbb{R}^3) \cap L^\nu(\mathbb{R}^3)$, we gain greater flexibility for the convolution estimates than a single variable exponent $L^{q(\cdot)}$ space allows.

\begin{theorem}\label{THM4 - Local - Existence}
Let $\alpha\in (\frac{1}{2},1]$, \(q \in \mathcal{P}^{\log}(\mathbb{R}^3)\) with $q^-=q_\infty>3$, \(p(\cdot) \in \mathcal{P}((0,1))\) with $p^->4$. If 
$\frac{2\alpha}{p^-}+\frac{3}{2}\left(\frac{2}{q^-}-\frac{1}{q^+}\right) < \alpha-\frac{1}{2}$ and 
\begin{enumerate}
\item $ f\in L^{1}\big(0,1;L^{q(\cdot)}(\mathbb{R}^3)\big)$ is a divergence-free exterior force, 
and
\item $ u_{0}\in L^{q(\cdot)}(\mathbb{R}^3)$ is a divergence-free initial data,
\end{enumerate}
 then there exists a time $T\in(0,1]$ and a unique mild solution of the fractional Navier-Stokes equations \eqref{NS_Intro}
in the space $L^{p(\cdot)}\big(0,T;L^{q(\cdot)}(\mathbb{R}^3)\big)$.
\end{theorem}
From the conditions established in this theorem, we derive the following lower bounds for the integrability exponents,  $p^-> \frac{4\alpha}{2\alpha-1}\geq 4$ and $q^-> \frac{3}{2\alpha-1}\geq 3$. 

We conclude this section with a local existence theorem established in the space $$L^{p(\cdot)}\big(0,T; L^{q(\cdot)}(\mathbb{R}) \cap L^2(\mathbb{R})\big).$$ By leveraging the estimates from Lemma~\ref{Lemma of L2 cap Lp} and invoking arguments analogous to those presented in the preceding proofs, we obtain the following result:
\begin{theorem} \label{THMFinall - Local - Existence}
Let $\alpha\in (\frac{3}{4},1]$, \(q \in \mathcal{P}^{\log}(\mathbb{R})\) with $q^-\geq 2$, $ q^-=q_\infty$ and \(p(\cdot) \in \mathcal{P}((0,1))\) with 
$\frac{2\alpha}{p^-}+\frac{1}{2}\max\left(\frac{1}{2},\frac{1+\vartheta}{2\alpha} \right) < \alpha-\frac{1}{2}$ where $\vartheta$ is calculated from (\ref{rho of Lp L2}) for $t\in(0,1]$. 
If 
\begin{enumerate}
\item $ f\in L^{1}\big(0,1;L^{q(\cdot)}(\mathbb{R})\cap L^2(\mathbb{R})\big)$ is a divergence-free exterior force, 
and
\item $ u_{0}\in L^{q(\cdot)}(\mathbb{R})\cap L^2(\mathbb{R})$ is divergence-free initial data,
\end{enumerate}
 then there exists a time $T\in(0,1]$ and a unique mild solution of the fractional Navier-Stokes equations \eqref{NS_Intro} 
in the space $L^{p(\cdot)}\big(0,T;L^{q(\cdot)}(\mathbb{R})\cap L^2(\mathbb{R})\big)$.
\end{theorem}
\begin{remark}\label{Rem 4.4}
In the previous theorem, the result where established in $\mathbb{R}$ rather than $\mathbb{R}^2$ or $\mathbb{R}^3$
 because if we have tried to prove analogous result  in $\mathbb{R}^2$ or $\mathbb{R}^3$ then we will arrive at the condition that
  $\frac{1}{2}\max(\frac{n}{2},\frac{1+n\vartheta}{2\alpha} )<\alpha-\frac{1}{2}$ ( $\vartheta$ is calculated from (\ref{rho of Lp L2}) for $t\in(0,1]$ )  which does not hold since  
   $\frac{1}{2}\max(\frac{n}{2},\frac{1+n\vartheta}{2\alpha}) \geq \frac{n}{4} \geq \frac{1}{2}$ if $n=2,3$ and 
 $\alpha-\frac{1}{2}\leq \frac{1}{2}$ if $\alpha\in (\frac{1}{2},1]$. Analogous results can be established in the space $L^{p(\cdot)}\big(0,T; L^{q(\cdot)}(\mathbb{R}) \cap L^2(\mathbb{R})\big)$ by leveraging the estimates from Lemma~\ref{Lemma of LV cap Lp}. In particular, for the space $L^{p(\cdot)}\big(0,T; L^{q(\cdot)}(\mathbb{R})\big)$, we examine the regime where $q^- = q_\infty = 2$.

\end{remark}

In this section, we have worked in a framework  where the exponent $p(t)$ is free from additional regularity constraints, such as the local $\log$-Hölder continuity condition. We observe that the requirements $\frac{2\alpha}{p^-} + \cdots < \alpha - \frac{1}{2}$ and $p^-> \frac{4\alpha}{2\alpha-1}$ can be relaxed to $\frac{\alpha}{p^-} + \cdots < \alpha - \frac{1}{2}$ and $p^-> \frac{2\alpha}{2\alpha-1}$ respectively by invoking the Riesz potential operator (see Definition~\ref{def:RieszPotential}). However, this gain in the integrability threshold comes at a cost, $p(t)$ must then satisfy the boundedness and local $\log$-Hölder continuity assumptions on $(0,1)$. We stress that, since the analysis is carried out in a variable exponent setting with minimal regularity requirements on $p(\cdot)$, the use of direct estimates provides a more flexible and robust approach for establishing the local existence result. In the sequel section, we apply the Riesz potential operator to establish global existence. The same approach can also be used to derive local existence results, under the additional structural assumptions on the exponent discussed above,  a more detailed discussion of this trade-off is provided in Remark~\ref{Rem 10}.

\bigskip

\subsection{Global Existence with Small initial Data}\label{sec-exist-glob-exist}
In the following, we consider two variable exponent $ p(t) \in \mathcal{P}(\,(0, +\infty)\,) $ and $q(\cdot)\in \mathcal{P}(\mathbb{R}^3)$, and the function space
\begin{equation*}
L^{p(\cdot)}\big((0,+\infty);L^{q(\cdot)}(\mathbb{R}^{3}) \cap L^\infty(\mathbb{R}^3)\big).
\end{equation*}
\begin{theorem}\label{THM glob 1}
Let $\alpha\in (\frac{1}{2},1]$, \(q \in \mathcal{P}^{\log}(\mathbb{R}^3)\) and \(p(\cdot) \in \mathcal{P}^{\log}([0,+\infty))\) with $p^-=p_\infty= \frac{2\alpha}{2\alpha-1}$ and $p^+<+\infty$. If 
$e_0\in L^{p(\cdot)}\big((0,+\infty);L^{q(\cdot)}(\mathbb{R}^{3})\cap L^\infty(\mathbb{R}^3)\big)$
 is small enough in norm then there is   a unique mild solution of the fractional Navier-Stokes equations \eqref{NS_Intro}
in the space $$L^{p(\cdot)}\big((0,+\infty);L^{q(\cdot)}(\mathbb{R}^{3}) \cap L^\infty(\mathbb{R}^3)\big).$$
\end{theorem}
\begin{proof}
Denote $\mathcal{E}_\infty^{p(\cdot),(q(\cdot),\infty)}:=L^{p(\cdot)}\big((0,+\infty);L^{q(\cdot)}(\mathbb{R}^{3})\cap L^\infty(\mathbb{R}^3)\big)$, we initiate the proof of global existence by utilizing the fundamental inequality (\ref{FromLocToGlob}), define 
\begin{equation*}
\Theta(t):=\left\|  \int_{0}^tK_{\alpha,t-s}\ast (u \otimes u)(s, x)\, ds\right\|_{L^{q(x)}(\mathbb{R}^3)\cap L^\infty(\mathbb{R}^3)} ,\, t\in (0,+\infty).
\end{equation*}
In view of the estimate established in (\ref{FromLocToGlob}), we obtain the following  estimate
\begin{equation*}
\Theta(t)\leq C \mathcal{R}_{\gamma}\left( \| u(s,\cdot) \|_{L^{q(\cdot)}(\mathbb{R}^3)\cap L^{\infty}(\mathbb{R}^3)}^2 \right)(t)
\end{equation*}
for every $t\in (0,+\infty)$, where $ \mathcal{R}_{\gamma}$ is the Riesz potential operator and  $\gamma :=1-\frac{1}{2\alpha}$. Define $ r(\cdot) \in \mathcal{P}(\,(0, +\infty)\,) $ by the relation
\begin{equation*}
\frac{1}{p(t)}=\frac{1}{r(t)}-\gamma \ ,\ t\in (0,+\infty).
\end{equation*}
It follows that $\gamma<\frac{1}{r^+}$ since $p^+$ is finite, by Theorem~\ref{Thm:RieszBoundedness} we have 
\begin{align*}
\|\Theta(t)\|_{L^{p(t)}((0,\infty))}&\leq 
\left\|\mathcal{R}_{\gamma}\left( \| u(s,\cdot) \|_{L^{q(\cdot)}(\mathbb{R}^3)\cap L^{\infty}(\mathbb{R}^n)}^2 \right)(t) \right\|_{L^{p(t)}((0,\infty))} \\
&\leq C \left\| \| u(t,\cdot) \|_{L^{q(\cdot)}(\mathbb{R}^3)\cap L^{\infty}(\mathbb{R}^n)}^2 \right\|_{L^{r(t)}((0,\infty))}.
\end{align*}
Since $p^-=p_\infty=\frac{1}{\gamma} $ then for every $t\in (0,+\infty)$, 
$$\frac{1}{r(t)}-\frac{1}{p(t)/2}= \gamma-\frac{1}{p(t)} \geq 0,$$
hence $r\leq \frac{p}{2}$, since \(p(\cdot) \in \mathcal{P}^{\log}((0,+\infty))\) and $p_\infty=\frac{1}{\gamma} $ 
then by Lemma~\ref{lem2},  \(1 \in L^{s(\cdot)}(\mathbb{R}^3)\) where 
\begin{equation*}
\frac{1}{s(t)} := \gamma-\frac{1}{p(t)}= \frac{1}{r(t)}-\frac{1}{p(t)/2}, \qquad \forall t\in (0,+\infty),
\end{equation*}
consequently  $L^{\frac{p(\cdot)}{2}}( [0,+\infty)) \hookrightarrow L^{r(\cdot)}( (0,+\infty))$. Therefore, by H\"{o}lder's inequality
\begin{align*}
\|\Theta(t)\|_{L^{p(t)}((0,\infty))}
&\leq C \left\| \| u(t,\cdot) \|_{L^{q(\cdot)}(\mathbb{R}^3)\cap L^{\infty}(\mathbb{R}^3)}^2 \right\|_{L^{r(t)}((0,\infty))} \\
&\leq C \left\| \| u(t,\cdot) \|_{L^{q(\cdot)}(\mathbb{R}^3)\cap L^{\infty}(\mathbb{R}^3)}^2 \right\|_{L^{\frac{p(t)}{2}}((0,\infty))}\\ 
&\leq C \left\| \| u(t,\cdot) \|_{L^{q(\cdot)}(\mathbb{R}^3)\cap L^{\infty}(\mathbb{R}^3)}\right\|_{L^{p(t)}((0,\infty))}\left\| \| u(t,\cdot) \|_{L^{q(\cdot)}(\mathbb{R}^3)\cap L^{\infty}(\mathbb{R}^3)} \right\|_{L^{p(t)}((0,\infty))}.
\end{align*}
Finally, we conclude that there exists a constant $C_B >0$ such that
\begin{align*}
\left\|B(\, u , \, u) \right\|_{\mathcal{E}_\infty^{p(\cdot),(q(\cdot),\infty)}}  =
\|\Theta(t)\|_{L^{p(t)}([0,\infty))} \leq C_B  \|\, u  \, \|_{\mathcal{E}_\infty^{p(\cdot),(q(\cdot),\infty)}} \ \|\, u  \, \|_{\mathcal{E}_\infty^{p(\cdot),(q(\cdot),\infty)}},
\end{align*} 
for every $u\in \mathcal{E}_\infty^{p(\cdot),(q(\cdot),\infty)}$. By choosing the initial data $e_0\in \mathcal{E}_\infty^{p(\cdot),(q(\cdot),\infty)}$ sufficiently small in the norm, a standard application of the Banach-Picard fixed-point theorem (Theorem~\ref{Thm:BanachPicard}) yields the existence of a unique global-in-time solution.

\end{proof}
\begin{remark}
Specifically, for the constant exponent $p = \frac{2\alpha}{2\alpha-1}$, Theorem~\ref{THM glob 1} provides a global-in-time existence result in the critical space $L^{\frac{2\alpha}{2\alpha-1}}\big((0,+\infty);L^{q(\cdot)}(\mathbb{R}^{3})\cap L^\infty(\mathbb{R}^3)\big)$, provided the initial data is sufficiently small in the norm.
\end{remark}
Similar arguments to those used for Theorem~\ref{THM glob 1} lead to the following theorem.
\begin{theorem}\label{THM glob 1-1}
Let $\alpha\in (\frac{1}{2},1]$, \(q \in (\frac{3}{2\alpha-1},+\infty]\) and \(p(\cdot) \in \mathcal{P}^{\log}((0,+\infty))\) with $p^-=p_\infty= \frac{2\alpha}{2\alpha-1-\frac{3}{q}}$ and $p^+<+\infty$. If 
$e_0\in L^{p(\cdot)}\big((0,+\infty);L^{q}(\mathbb{R}^{3})\big)$
 is small enough in norm  then there is   a unique mild solution of the fractional Navier-Stokes equations \eqref{NS_Intro}
in the space $$L^{p(\cdot)}\big((0,+\infty);L^{q}(\mathbb{R}^{3})\big).$$
\end{theorem}

In the following theorems we will deal with the space $L^{p(\cdot)}\big(0,T;L^{q(\cdot)}(\mathbb{R}^{3})\big)$, the analysis is based on the estimates from Lemmas~\ref{LemmaCom 1}, \ref{LemmaCom 2}, \ref{Lemma of LV cap Lp}, \ref{Lemma of L2 cap Lp}, we denote the

In the estimate established in (\ref{FromLocToGlob}) the exponent  over $(t-s)$ which is $\frac{1}{2\alpha}$ is constant 
on the whole time  interval but in the space when we deal with the spaces  $L^{q(\cdot)}(\mathbb{R}^{3})$ or 
$L^{q(\cdot)}(\mathbb{R}^{3})\cap L^{\nu}(\mathbb{R}^{3})$  this exponent varies over time, by virtue of the estimate (\ref{rho-label}) from  Lemma~\ref{LemmaCom 1} the exponent is be defined by  $\vartheta_1$ on  $(0,1]$ and  $\vartheta_2$ otherwise,  where
\begin{align*}
\vartheta_1:=\Big(\frac{2}{q^-}-\frac{1}{q_\infty} \Big)
 \Big(\frac{2+q_\infty(1-\frac{2}{q^-})}{1+q_\infty(1-\frac{2}{q^-})}- \frac{1}{q_\infty-1}\Big) , \,
\vartheta_2:=\frac{1}{q_\infty}
\Big(\frac{q_\infty(1-\frac{2}{q^-})}{1+q_\infty(1-\frac{2}{q^-})}+ \frac{1}{q_\infty-1}\Big).
\end{align*}

\begin{theorem}\label{THM glob 2}
Let $\alpha\in (\frac{1}{2},1]$, \(q \in \mathcal{P}^{\log}(\mathbb{R}^3)\) with $\frac{3}{2\alpha-1}<q^-\leq q^+=q_\infty <+\infty$, 
 and $p \in \mathcal{P}^{\log}((0,+\infty))$ with $\frac{2\alpha}{2\alpha-1}<p^-\leq p^+<+\infty$. Suppose that   $\frac{\alpha}{p^-} +\frac{3\vartheta_1}{2}\leq \alpha -\frac{1}{2}$. If  $e_0\in L^{p(\cdot)}\big(0,T;L^{q(\cdot)}(\mathbb{R}^{3})\big)$
 is small enough in norm then there  a unique mild solution of the fractional Navier--Stokes equations  \eqref{NS_Intro}
in the space
 $$L^{p(\cdot)}\big(0,T;L^{q(\cdot)}(\mathbb{R}^{3})\big)$$
\end{theorem}
\begin{proof}The proof is similar to that of Theorem \ref{THM glob 1}, let 
\begin{equation*}
\Theta(t):=\left\|  \int_{0}^tK_{\alpha,t-s}\ast (u \otimes u)(s, x)\, ds\right\|_{L^{q(x)}(\mathbb{R}^3)} ,\, t\in (0,T).
\end{equation*}
 Then, by similar arguments as in the proof of  Theorem \ref{THM1 - Local - Existence}  we obtain a similar
  estimation to (\ref{enequa THM1}) as follows, for every almost  every $t\in (0,T)$ we have
\begin{align*}
\Theta(t)\leq C \int_{0}^{t}\frac{\| u(s, x) \|_{L^{q(\cdot)}(\mathbb{R}^3)}^2}{(t-s)^{\frac{1+3\vartheta(t-s)}{2\alpha}}}  \,ds,
\end{align*}
where $\vartheta$ is from (\ref{rho-label}), denote the constants $\vartheta_1$ and $\vartheta_2$ by 
 $\vartheta_1:=\vartheta(t), t \in (0,1]$ and $ \vartheta_2:=\vartheta(t), t >1$ , Without loss of generality, we may assume that $T \geq 1$. For every $t\in (0,1]$ we have 
\begin{equation*}
\Theta(t)\leq \mathcal{R}_{\gamma_1}\left( \| u(s,\cdot) \|_{L^{q(\cdot)}(\mathbb{R}^n)}^2\chi_{(0,T)} \right)(t)
,\,  
\end{equation*}
where $\gamma_1:=1-\frac{1+3\vartheta_1}{2\alpha}$.  By the assumptions of the theorem, $0\leq 3 \vartheta_1 <2\alpha-1$, it follows that $0<\gamma_1 <1$. 
For every $t\in(0,T)$,
\begin{align}
\int_{0}^{t}\frac{ \| u(s, x) \|_{L^{q(\cdot)}(\mathbb{R}^3)}^2}{(t-s)^{\frac{1+3\vartheta(t-s)}{2\alpha}}} \,ds 
&= \int_{0}^{t-1}\frac{ \| u(s, x) \|_{L^{q(\cdot)}(\mathbb{R}^3)}^2}{(t-s)^{\frac{1+3\vartheta_2}{2\alpha}}} \,ds +
\int_{t-1}^{t}\frac{ \| u(s, x) \|_{L^{q(\cdot)}(\mathbb{R}^3)}^2}{(t-s)^{\frac{1+3\vartheta_1}{2\alpha}}} \,ds. \label{La-Est}\\
&\leq  \int_{0}^{T} \| u(s, x) \|_{L^{q(\cdot)}(\mathbb{R}^3)}^2 \,ds +
\int_{t-1}^{t}\frac{ \| u(s, x) \|_{L^{q(\cdot)}(\mathbb{R}^3)}^2}{(t-s)^{\frac{1+3\vartheta_1}{2\alpha}}} \,ds.\nonumber
\end{align}
We conclude the following estimate, for almost  every $t\in(0,T)$,
\begin{equation*}
\Theta(t)\leq C  \mathcal{R}_{\gamma_1}\left( \| u(s,\cdot) \|_{L^{q(\cdot)}(\mathbb{R}^n)}^2 \chi_{(0,T)}\right)(t)+
\int_{0}^{T} \| u(s, x) \|_{L^{q(\cdot)}(\mathbb{R}^3)}^2 \,ds ,
\end{equation*}
to estimate the second term on the left side we apply H\"{o}lder's inequality which yields 
$$\int_{0}^{T} \| u(s, x) \|_{L^{q(\cdot)}(\mathbb{R}^3)}^2 \,ds \leq C 
\left\| \| u(t,\cdot) \|_{L^{q(\cdot)}(\mathbb{R}^3)}^2 \right\|_{L^{\frac{p(t)}{2}}((0,T))}
\leq C 
\left\| \| u(t,\cdot) \|_{L^{q(\cdot)}(\mathbb{R}^3)} \right\|_{L^{p(t)}((0,T))}^2,
$$
and with similar arguments as in the proof of Theorem~\ref{THM glob 1} we estimate the first term on the left  side as follows 
\begin{align*}
\left\|\mathcal{R}_{\gamma_1}\left( \| u(s,\cdot) \|_{L^{q(\cdot)}(\mathbb{R}^n)}^2\chi_{(0,T)} \right)(t) \right\|_{L^{p(t)}((0,\infty))} 
&\leq C \left\| \| u(t,\cdot) \|_{L^{q(\cdot)}(\mathbb{R}^3)}^2\chi_{(0,T)} \right\|_{L^{r(t)}((0,\infty))}\\
&= C \left\| \| u(t,\cdot) \|_{L^{q(\cdot)}(\mathbb{R}^3)}^2 \right\|_{L^{r(t)}((0,T))}
\end{align*}
where  $ r(\cdot) \in \mathcal{P}(\,(0,\infty)\,) $ is defined by the relation $\frac{1}{p(t)}=\frac{1}{r(t)}-\gamma_1 \ ,\ t\in (0,\infty).$ Since  $p^-\geq\frac{1}{\gamma_1}$, it follows that $r\leq \frac{p}{2}$, therefore $L^{\frac{p(\cdot)}{2}}( (0,T)) \hookrightarrow L^{r(\cdot)}( (0,T))$, we conclude that 
\begin{align*}
\|\Theta(t)\|_{L^{p(t)}((0,T))}&\leq 
\left\|\mathcal{R}_{\gamma_1}\left( \| u(s,\cdot) \|_{L^{q(\cdot)}(\mathbb{R}^3)}^2\chi_{(0,T)} \right)(t) \right\|_{L^{p(t)}((0,\infty))} + \left\| \| u(t,\cdot) \|_{L^{q(\cdot)}(\mathbb{R}^3)} \right\|_{L^{p(t)}((0,T))}^2 \\
&\leq C \left\| \| u(t,\cdot) \|_{L^{q(\cdot)}(\mathbb{R}^3)} \right\|_{L^{p(t)}((0,T))}^2.
\end{align*}
Finally we conclude that there exists a constant $C_B >0$ such that, for every $u\in L^{p(\cdot)}\big(0,T;L^{q(\cdot)}(\mathbb{R}^{3})\big)$,
\begin{align*}
\left\|B(\, u , \, u) \right\|_{L^{p(\cdot)}(0,T;L^{q(\cdot)}(\mathbb{R}^{3}))}  &=
\|\Theta(t)\|_{L^{p(t)}((0,T))} \\
&\leq C_B  \|\, u  \, \|_{L^{p(\cdot)}(0,T;L^{q(\cdot)}(\mathbb{R}^{3}))} \ \|\, u  \, \|_{L^{p(\cdot)}(0,T;L^{q(\cdot)}(\mathbb{R}^{3}))}.
\end{align*} 
The existence and uniqueness of a global-in-time solution then follow from the Banach-Picard fixed-point theorem, provided the initial data is sufficiently small in the norm.

\end{proof}

\begin{remark}\label{Rem 10}\ 
\begin{enumerate}
\item While Theorem~\ref{THM glob 2} is proved here using the Riesz potential operator, the global existence result could alternatively be established through  direct calculations analogous to the proof of Theorem~\ref{THM1 - Local - Existence}.  In that setting, the global existence would follow for $p \in \mathcal{P}((0,T))$, provided the following condition $\frac{2\alpha}{p^-} +\frac{3\vartheta_1}{2}< \alpha -\frac{1}{2}$ is satisfied with $\frac{2\alpha}{p^-}$ instead of $\frac{\alpha}{p^-}$. Notably, this alternative approach offers the distinct advantage of lifting the boundedness requirement on $p$. This allows for a more general class of exponents $p(\cdot)$ that may grow indefinitely.

\item If $p(\cdot)$ is bounded and  $\frac{1}{p(\cdot)}$  is   locally $\log$-H\"{o}lder continuous on $(0,T)$ where $T >0$ then $\lim_{t\to T^-}\frac{1}{p(t)}:=\iota$ exists and is positive, therefore by setting $\widetilde{p}(t):=p(t)$ on $(0,T)$ and $\widetilde{p}(t):=\iota^{-1}$ on $[T,+\infty)$  then  the exponent $\widetilde{p}(\cdot)$ is an extension of $p(\cdot)$ which is \textit{globally} $\log$-H\"{o}lder continuous on $(0,+\infty)$, that is $\widetilde{p}(t) \in \mathcal{P}^{\log}((0,+\infty))$. Based on this observation, Theorem~\ref{THM glob 2} remains valid provided that $p(\cdot)$ is bounded and $\frac{1}{p(\cdot)}$ is   locally $\log$-H\"{o}lder continuous on $(0,T)$. A more general  extension property for the exponent function is stated in \cite[Lemma 2.4]{CF} and \cite[Proposition 4.1.7]{DHHR}. 

\item The Riesz potential operator approach also provides an effective framework for proving the local existence of mild solutions. In this setting, the requirements deviate slightly from those in Theorem~\ref{THM1 - Local - Existence}, specifically, we assume that $p(\cdot)$ is \textit{bounded} and $\frac{1}{p(\cdot)}$ is locally log-Hölder continuous on $(0,1)$. Under these assumptions, the condition 
$\frac{\alpha}{p^-}+\frac{3\vartheta _1}{2} < \alpha-\frac{1}{2}$ is sufficient to obtain the same local existence result. 
Thus, the Riesz potential method yields an improved integrability threshold compared to the direct approach used in Section~\ref{Sec-Exis-local}, where a stronger restriction involving $2\alpha$ appears, but without requiring boundedness or $\log$-Hölder continuity of the exponent $p(\cdot)$. Similar alternative versions of the results in Section~\ref{Sec-Exis-local} can be derived within the Riesz potential framework under these stronger structural assumptions on $p(\cdot)$.
\item Theorem~\ref{THM1 - Local - Existence} was established using the Riesz potential operator approach in the paper  \cite[Theorem 1]{GVH2}. However, that result was situated in the mixed-norm space $L^{p(\cdot)}\big(0,T; L^{q}(\mathbb{R}^{3})\big)$ with a constant  exponent $q$, this choice was necessitated by the current lack of general convolution-type estimates in fully variable exponent spaces.
\end{enumerate}
\end{remark}

\begin{remark}\label{finalRamrk}
Following the methodology of Theorem~\ref{THM glob 2}, we establish corresponding global existence results for sufficiently small initial data within the space $L^{p(\cdot)}\big(0,T;L^{q(\cdot)}(\mathbb{R}^{3})\big)$ with $q_\infty=+\infty$ or $q^-=q_\infty$, or the space 
$L^{p(\cdot)}\big(0,T;L^{q(\cdot)}(\mathbb{R}^{3})\cap L^{\nu}(\mathbb{R}^{3})\big)$ 
where we apply the estimates from Lemmas~\ref{LemmaCom 2}, \ref{Lemma of LV cap Lp}, \ref{Lemma of L2 cap Lp}.

In Theorem~\ref{THM glob 1} the global Existence was proved by virtue of the embedding $L^{\frac{p(\cdot)}{2}}( (0,+\infty)) \hookrightarrow L^{r(\cdot)}( (0,+\infty))$ where  $\frac{1}{p(t)}=\frac{1}{r(t)}-\gamma$, this embedding holds if $p^-=p_\infty=\frac{1}{\gamma}$. Attempting to extend this framework to obtain a global-in-time solution in the  space
\begin{align*}
L^{p(\cdot)}  \big((0,+\infty);L^{q(\cdot)}(\mathbb{R}^{3})\big)
\end{align*}
presents significant technical challenges. Specifically,
the estimate (\ref{La-Est}) involve two distinct exponents, $\alpha_1=\frac{1+3\vartheta_1}{2\alpha}$ on $(0,1]$ and $ \alpha_2=\frac{1+3\vartheta_2}{2\alpha}$ on $(1,+\infty)$, so the global existence holds on $(0,+\infty)$ if we estimate both terms in 
$ L^{p(t)}((0,+\infty)) $, the second term may be bounded by Riez operator via the estimate $\int_{t-1}^{t}\frac{\cdots}{(t-s)^{\frac{1+3\vartheta_1}{2\alpha}}} \,ds \leq \mathcal{R}_{\gamma_1}(\cdots)(t)$ where $\gamma=1-\alpha_1$, now if $p^-=p_\infty=\frac{1}{\gamma_1}$ we can use achieve the desired estimate as in the proof of Theorem~\ref{THM glob 1}, the second term 
$\int_{0}^{t-1}\frac{ \cdots}{(t-s)^{\frac{1+3\vartheta_2}{2\alpha}}} \,ds$ may be bounded by $\mathcal{R}_{\gamma_2}$
 with $\gamma=1-\alpha_2$, but to apply the previous technique  we need $p^-=p_\infty=\frac{1}{\gamma_2}$ which implies  that $\gamma_1=\gamma_2$ which unfortunately doesn't hold since $\vartheta_2<\vartheta_1 $ unless the exponent $q$ is constant. We note a possible alternative estimate for the second term  as follows,
$$\int_{0}^{t-1}\frac{ \| u(s, x) \|_{L^{q(\cdot)}(\mathbb{R}^n)}^2}{(t-s)^{\frac{1+3\vartheta_2}{2\alpha}}} \,ds \leq C
\eta_{1,\frac{1+3\vartheta_2}{2\alpha}}*\left( \| u(s, x) \|_{L^{q(\cdot)}(\mathbb{R}^n)}^2\right)(t),
$$ 
however, because $\frac{1+3\vartheta_2}{2\alpha}<\frac{1+3\vartheta_1}{2\alpha}<1$ in general, this approach is not amenable to establishing additional global-in-time estimates in $(0,+\infty)$ using our current techniques. 
\end{remark}


\bigskip
\begin{thebibliography}{99}

\bibitem{AH}
Almeida, A., H\"{a}st\"{o}, P.: Besov spaces with variable smoothness and integrability. \textit{J. Funct. Anal.} \textbf{258}(5), 1628--1655 (2010)

\bibitem{Chamorro2025}
Chamorro, D., Vergara-Hermosilla, G.: Liouville-type theorems for stationary Navier--Stokes equations with Lebesgue spaces of variable exponent. \textit{Doc. Math.} (2025). \href{https://doi.org/10.4171/DM/1018}{doi:10.4171/DM/1018}

\bibitem{GVH2}
Chamorro, D., Vergara-Hermosilla, G.: Lebesgue spaces with variable exponent: some applications to the Navier--Stokes equations. \textit{Positivity} \textbf{28}, 24 (2024)

\bibitem{ChangJin2016}
Chang, T., Jin, B.J.: Initial and boundary values for $L^q_\alpha(L^p)$ solution of the Navier--Stokes equations in the half-space. \textit{J. Math. Anal. Appl.} \textbf{439}(1), 1--25 (2016)

\bibitem{CVZ}
Chen, H., Vergara-Hermosilla, G., Zhao, J.: Well-posedness of the 2D surface quasi-geostrophic equation in variable Lebesgue spaces. \textit{Rend. Circ. Mat. Palermo, II. Ser} \textbf{74}, 134 (2025)


\bibitem{CF}
Cruz-Uribe, D., Fiorenza, A.: \textit{Variable Lebesgue Spaces. Foundations and Harmonic Analysis}. Applied and Numerical Harmonic Analysis. Birkh\"{a}user/Springer, Heidelberg (2013)


\bibitem{DHHR}
Diening, L., Harjulehto, P., H\"{a}st\"{o}, P., R\r{u}\v{z}i\v{c}ka, M.: \textit{Lebesgue and Sobolev Spaces with Variable Exponents}. Lecture Notes in Mathematics, vol. 2017. Springer-Verlag, Berlin (2011)



\bibitem{DHR}
Diening, L., H\"{a}st\"{o}, P., Roudenko, S.: Function spaces of variable smoothness and integrability. \textit{J. Funct. Anal.} \textbf{256}(6), 1731--1768 (2009)

\bibitem{DieningHastoNekvinda2005}
Diening, L., H\"{a}st\"{o}, P., Nekvinda, A.: Open problems in variable exponent Lebesgue and Sobolev spaces. In: \textit{Proc. Int. Conf. Function Spaces, Differential Operators and Nonlinear Analysis (FSDONA04)}, pp. 38--58. Math. Inst. Acad. Sci. Czech Rep., Prague (2005)

\bibitem{DD}
Dong, H., Du, D.: The Navier--Stokes equations in the critical Lebesgue space. \textit{Commun. Math. Phys.} \textbf{292}, 811--827 (2009)

\bibitem{DudekSkrzypczak2017}
Dudek, S., Skrzypczak, I.: Liouville theorems for elliptic problems in variable exponent spaces. \textit{Commun. Pure Appl. Anal.} \textbf{16}(2), 513--532 (2017)

\bibitem{Jiang2015}
Jiang, X.: Navier--Stokes equations with variable viscosity in variable exponent spaces of Clifford-valued functions. \textit{Bound. Value Probl.} \textbf{2015}, 30 (2015)

\bibitem{KatoFujita1962}
Kato, T., Fujita, H.: On the nonstationary Navier--Stokes system. \textit{Rend. Sem. Mat. Univ. Padova} \textbf{32}, 243--260 (1962)

\bibitem{KozonoShimizu2018}
Kozono, H., Shimizu, S.: Strong solutions of the Navier--Stokes equations based on the maximal Lorentz regularity theorem in Besov spaces. \textit{J. Funct. Anal.} \textbf{274}(5), 1328--1378 (2018)

\bibitem{LemarieRieusset2002}
Lemari\'{e}-Rieusset, P.G.: \textit{Recent Developments in the Navier--Stokes Problem}. CRC Press, Boca Raton (2002)

\bibitem{Lemarie-Rieusset2024}
Lemari\'{e}-Rieusset, P.G.: \textit{The Navier--Stokes Problem in the 21st Century}, 2nd edn. Chapman and Hall/CRC, Boca Raton (2023)

\bibitem{LiZhangZhang2021}
Li, Y., Zhang, J., Zhang, T.: Asymptotic stability of Landau solutions to Navier--Stokes system under $L^p$-perturbations. \textit{J. Differ. Equ.} \textbf{299}, 114--144 (2021)

\bibitem{Miao2008}
Miao, C., Yuan, B., Zhang, B.: Well-posedness of the Cauchy problem for the fractional power dissipative equations. \textit{Nonlinear Anal.} \textbf{68}, 461--484 (2008)

\bibitem{MihailescuRadulescu2012}
Mih\v{a}ilescu, M., R\v{a}dulescu, V.: On a PDE involving the $\mathcal{A}_{p(\cdot)}$-Laplace operator. \textit{Nonlinear Anal.} \textbf{75}(12), 4616--4628 (2012)

\bibitem{Phan2019}
Phan, T.: Well-posedness for the Navier--Stokes equations in critical mixed-norm Lebesgue spaces. \textit{J. Evol. Equ.} \textbf{19}, 1131--1162 (2019)

\bibitem{QW2023}
Qian, C., Wang, L.: Asymptotic profiles and concentration-diffusion effects in fractional incompressible flows. \textit{Nonlinear Anal.} \textbf{228}, 113185 (2023)


\bibitem{Samak1996}
Samko, S.G.: Convolution and potential type operators in $L^{p(x)}(\mathbb{R}^n)$. \textit{Integr. Transform. Spec. Funct.} \textbf{7}(3--4), 261--284 (1998). \href{http://dx.doi.org/10.1080/10652469808819204}{doi:10.1080/10652469808819204}

\bibitem{Sun2025}
Sun, J., Mai, Y., Yang, M.: Well-posedness of generalized magnetohydrodynamic equations in variable Lebesgue spaces. \textit{Electron. J. Differ. Equ.} \textbf{2025}(111), 1--15 (2025). \href{https://doi.org/10.58997/ejde.2025.111}{doi:10.58997/ejde.2025.111}

\bibitem{GVH1}
Vergara-Hermosilla, G.: Remarks on variable Lebesgue spaces and fractional Navier--Stokes equations. \textit{ESAIM Proc. Surv.} \textbf{79}, 110--125 (2025). \href{https://doi.org/10.1051/proc/202579110}{doi:10.1051/proc/202579110}
\end{thebibliography}
\end{document}